\documentclass[a4paper,12pt,notitlepage]{amsart}
\usepackage[utf8]{inputenc}
\usepackage{amsmath}
\usepackage{amsthm}
\usepackage{amsfonts}
\usepackage{amssymb}
 \usepackage{amsrefs}
 \usepackage{enumerate}
\usepackage{graphicx}

\usepackage{verbatim}

\usepackage{hyperref}

\newcommand{\bN}{\mathbb{N}}
\newcommand{\bR}{\mathbb{R}}
\newcommand{\bZ}{\mathbb{Z}}
\newcommand{\bP}{\mathbb{P}}

\newcommand{\cA}{\mathcal{A}}
\newcommand{\cB}{\mathcal{B}}

\newcommand{\cS}{\mathcal{S}}

\newcommand{\rmd}{\mathrm{d}}
\newcommand{\rmi}{\mathrm{i}}

\newcommand{\bfc}{\boldsymbol{c}}

\newcommand{\Ds}[1][N]{D^{\ast}_{#1}}

\newcommand{\e}{e}
\newcommand{\eps}{\varepsilon}

\DeclareMathOperator{\id}{id}
\theoremstyle{plain}		\newtheorem{theorem}{Theorem}[section]
						\newtheorem{lemma}[theorem]{Lemma}
\theoremstyle{definition}	\newtheorem{proposition}[theorem]{Proposition}
						\newtheorem{remark}[theorem]{Remark}

\title[Perturbed Halton--Kronecker sequences]{Sharp general and metric bounds for the star discrepancy of {perturbed} Halton--Kronecker sequences}
\author{Roswitha Hofer}
\thanks{The  authors are supported by the Austrian Science Fund (FWF), Projects F5505-26 and F5507-26, which are part of the Special Research Program ``Quasi-Monte Carlo Methods: Theory and Applications''}
\address{Institute of Fincancial Mathematics and Applied Number Theory, Johannes Kepler University  Linz}
\email{roswitha.hofer@jku.at}
\author{Florian Puchhammer}
\address{Institute of Fincancial Mathematics and Applied Number Theory, Johannes Kepler University  Linz}
 \email{florian.puchhammer@jku.at}
 

\begin{document}

\begin{abstract}

We consider the star discrepancy of two-dimensional sequences made up as a hybrid between a Kronecker sequence and a perturbed Halton sequence in base 2, where the perturbation is achieved by a digital-sequence construction in the sense of Niederreiter whose generating matrix contains a periodic perturbing sequence of a given period length. Under the assumption that the Kronecker sequence involves a parameter with bounded continued fraction coefficients sharp discrepancy estimates are obtained. Furthermore, we study the problem from a metric point of view as well. Finally, we also present sharp general and tight metric bounds for certain lacunary trigonometric products which appear to be strongly related to these problems.
\end{abstract}

\maketitle

\section{Introduction and statement of the results}

We investigate distribution properties of \emph{perturbed Halton--Kronecker sequences}, i.e., two-dimensional hybrid sequences $(z_k(n))_{k\geq0}$ of the form
$$
z_k(n)=\left( x_k(n),\{k\alpha\} \right),
$$
where $(\{k\alpha\})_{k\geq0}$ denotes the \emph{Kronecker sequence} with (irrational) parameter $\alpha$ and where $(x_k(n))_{k\geq0}$ is a \emph{perturbed Halton sequence} in base 2. The latter is a special instance of a digital sequence in the sense of Niederreiter (\cites{Nie87}) and is constructed as follows.

For the construction of a more generic sequence $(x_k)_{k\geq0}$ we fix an infinite matrix $C$ over $\{0,1\}$, a so-called \emph{generating matrix}, as the identity whose first row is perturbed by a sequence $\bfc=(c_0,c_1,c_2,\ldots)$ in $\{0,1\}$. More precisely,
\begin{equation}
\label{eqn:C}
 C=
 \begin{pmatrix}
  c_0		&c_1	&c_2	&\cdots \\
  0		&1		&0		&\cdots \\
  0		&0		&1		&\ddots \\
  \vdots	&\vdots	&\ddots 	&\ddots
 \end{pmatrix}.
\end{equation}
Furthermore, for each non-negative integer $k$ we assemble the dyadic digits of its binary expansion $k_0+k_12+k_22^2+\cdots$ into the vector $(k_0,k_1,k_2,\ldots)=:\boldsymbol{k}$ and compute $(y_0,y_1,y_2,\ldots)=C\cdot \boldsymbol{k}^{\top}$ modulo $2$. 
Subsequently, we define the $k$th element of our digital sequence $(x_k)_{k\geq 0}$ as
\begin{equation*}
 x_k=\frac{y_0}{2}+\frac{y_1}{2^2}+\frac{y_2}{2^3}+\cdots.
\end{equation*}
Taking the \emph{perturbing sequence} in the special form
\begin{equation}
\label{eqn:c}
 \bfc=(\underbrace{10\ldots0}_{n}\underbrace{10\ldots0}_{n}\ldots)
\end{equation}
with period length $n$ yields the sought sequence $(x_k(n))_{k\geq0}$.

We intend to use perturbed Halton--Kronecker sequences to \emph{approximate} uniform distribution on the  unit square $[0,1)^{2}$. The \emph{star discrepancy} serves as a quality measure for how evenly such a sequence is distributed. For the first $N$ elements of a sequence $\cS=(s_0,s_1,\ldots)$ in $[0,1)^{2}$ it is defined as
\begin{equation*}
 \Ds(\cS)=\sup_{\boldsymbol{x}=(x_1,x_2)\in(0,1]^{2}}\frac{1}{N}\Big|\cA_N(\cS,[\boldsymbol{0},\boldsymbol{x}))-N\lambda([\boldsymbol{0},\boldsymbol{x}))    \Big|,
\end{equation*}
where $\lambda([\boldsymbol{0},\boldsymbol{x}))$ denotes the two-dimensional Lebesgue measure of the box $[\boldsymbol{0},\boldsymbol{x})=[0,x_1)\times[0,x_2)$ and where
\begin{equation*}
 \cA_N(\cS,[\boldsymbol{0},\boldsymbol{x}))=\#\Big(\{s_n:~0\leq n<N\}\cap[\boldsymbol{0},\boldsymbol{x})\Big).
\end{equation*}
counts the number of elements of the initial segment of $\cS$ of size $N$ which lie in  $[\boldsymbol{0},\boldsymbol{x})$. If it is clear from the context which sequence we consider, we may omit the respective argument. Certainly, this entity can be extended to unanchored boxes and higher dimensions as well.
For an extensive survey on $\Ds$ and the sequences involved we refer to the books \cites{DicDig10,MatGeo99,NieRan92}.

Before we present the main results of this paper we require some notation. In what follows we write $A(N)\ll_{X} B(N)$ if $|A(N)|\leq c_{X} |B(N)|$ for all $N$ large enough and $A(N)\gg_{X} B(N)$ if $|A(N)|\geq c_{X} |B(N)|$ for infinitely many $N\in\bN$; here $c_{X}>0$ is a constant exclusively depending on a collection of parameters indicated by $X$.
First of all, we consider perturbed Halton--Kronecker sequences in the case where $\alpha$ has bounded continued fraction coefficients. 
\begin{theorem}
 \label{thm:genupper}
 Let $n\in\bN$ and $\alpha\in(0,1)$ with bounded continued fraction coefficients. Then the star discrepancy of the first $N$ elements of the sequence $(z_k(n))_{k\geq 0}$ satisfies
 \begin{equation*}
  N \Ds(z_k(n))\ll_{n} N^{a(n)+\eps}
 \end{equation*}
for all $\eps>0$, where
\begin{equation}
\label{eqn:an}
 a(n)=\log_{2^n}\left( \cot\frac{\pi}{2(2^n+1)}\right).
\end{equation}
\end{theorem}
On the other hand, we can show that this bound is essentially sharp by utilizing a special $\beta$, as introduced by Shallit \cite{ShaSim79}, which has both bounded continued fraction coefficients and an explicitly known dyadic expansion. 

\begin{theorem}
 \label{thm:genlower}
 Let $n\in\bN$ and let $\alpha=\frac{2^n}{2(2^n+1)}+\beta$ with $\beta=\sum_{k\geq0}4^{-2^{k}}$. Then we have
 \begin{equation*}
  N \Ds(z_k(n))\gg N^{a(n)-\eps}
 \end{equation*}
for all $\eps>0$, where $a(n)$ is given by (\ref{eqn:an}).
\end{theorem}
As a matter of fact, Theorem~\ref{thm:genupper} holds for a larger class of $\alpha$, i.e., for $\alpha$ of some finite type $\sigma\geq1$. Details on $\sigma$ can be found in Remark~\ref{rem:41} after the proof of the theorem. The primary interest, however, lies in $\alpha$'s with bounded continued fraction coefficients, since the Kronecker component satisfies an optimal discrepancy bound in this case.

\begin{remark}\label{rem:1}
In the limit case $n=\infty$, i.e. $C$ is the identity, $(x_k(n))_{k\geq0}$ becomes the pure Halton sequence. The \emph{Halton--Kronecker sequence} $(z_k(\infty))_{k\geq0}$ was originally studied by Niederreiter \cite{NieOnt09} and, recently,  by the first author together with Larcher and Drmota \cite{drmlarHK}, who obtained
 \begin{equation*}
  N \Ds(z_k(\infty))\ll_\alpha N^{1/2}\log N \ll_\eps N^{1/2+\eps}
 \end{equation*}
for every $\alpha\in(0,1)$ with bounded continued fraction coefficients and all $\eps>0$ (see also \cite{NieImp12}). Furthermore, for $\alpha=\sum_{k\geq0}4^{-2^{k}}$ the following inequality holds
 \begin{equation*}
  N \Ds(z_k(\infty))\gg N^{1/2}.
 \end{equation*}
\end{remark}

The lower bound of Theorem~\ref{thm:genlower} is in close connection to  one-dimensional subsequences of the pure Kronecker sequence, i.e., $(\{m_k\alpha\})_{k\geq0}$.  It is easily seen that \emph{evil Kronecker sequences}, which were studied by the first author together with Aistleitner and Larcher in \cite{AisOnp15} and are denoted by $(\{e_k \alpha\})_{k\geq 0}$, are directly linked to $(z_k(1))_{k\geq 0}$. Several techniques of our proof reach back to this paper. Here, the sequence of \emph{evil numbers} $(e_k)_{k\geq0}$ denotes the increasing sequence of non-negative integers whose sum of dyadic digits is even. Similarly, it turns out in the proof of Theorem~\ref{thm:genlower} that the sequence  $(m_{k})_{k\geq 0}$ related to  $(z_k(n))_{k\geq 0}$ is the increasing sequence of non-negative integers with an even sum of digits in base $2^{n}$, i.e.,
\begin{equation}
\label{eqn:mk}
 m_k=\mu_0+2\mu_1+2^2\mu_2+\cdots,~\mu_i\in\{0,1\},\text{ with }\mu_0+\mu_n+\mu_{2n}+\cdots\equiv 0\bmod{2}.
\end{equation}

Concerning the sharp exponent $a(n)$ in  Theorem~\ref{thm:genupper} and Theorem~\ref{thm:genlower} above some remarks are in order.  Prior to this paper, two results for specific $n$ are known to the authors, namely $n=1$ (see~\cite{AisOnp15}) and $n=\infty$ (see Remark~\ref{rem:1}). In the first case an exponent of $\log_4 3\approx0.79\ldots$ is obtained. Apparently, this coincides with $a(1)$. Hence, the current paper can be seen as an extension of \cite{AisOnp15}. In the second case, i.e. $n=\infty$, Remark~\ref{rem:1} states an exponent of $1/2$. Hence, naturally one would expect $a(n)$ to decrease from $\log_4 3$ to $1/2$. Surprisingly, the opposite is the case: $a(n)$ increases w.r.t. $n$. This means that if the density of $1$'s in the first row of our generating matrix $C$ decreases, the best possible bound for the  star discrepancy of the hybrid sequence grows. Figure~\ref{fig:an} shows a plot of $a(n)$ for $1\leq n\leq50$.

\begin{figure}[!hbtp]
\centering
\includegraphics[width=0.75\textwidth]{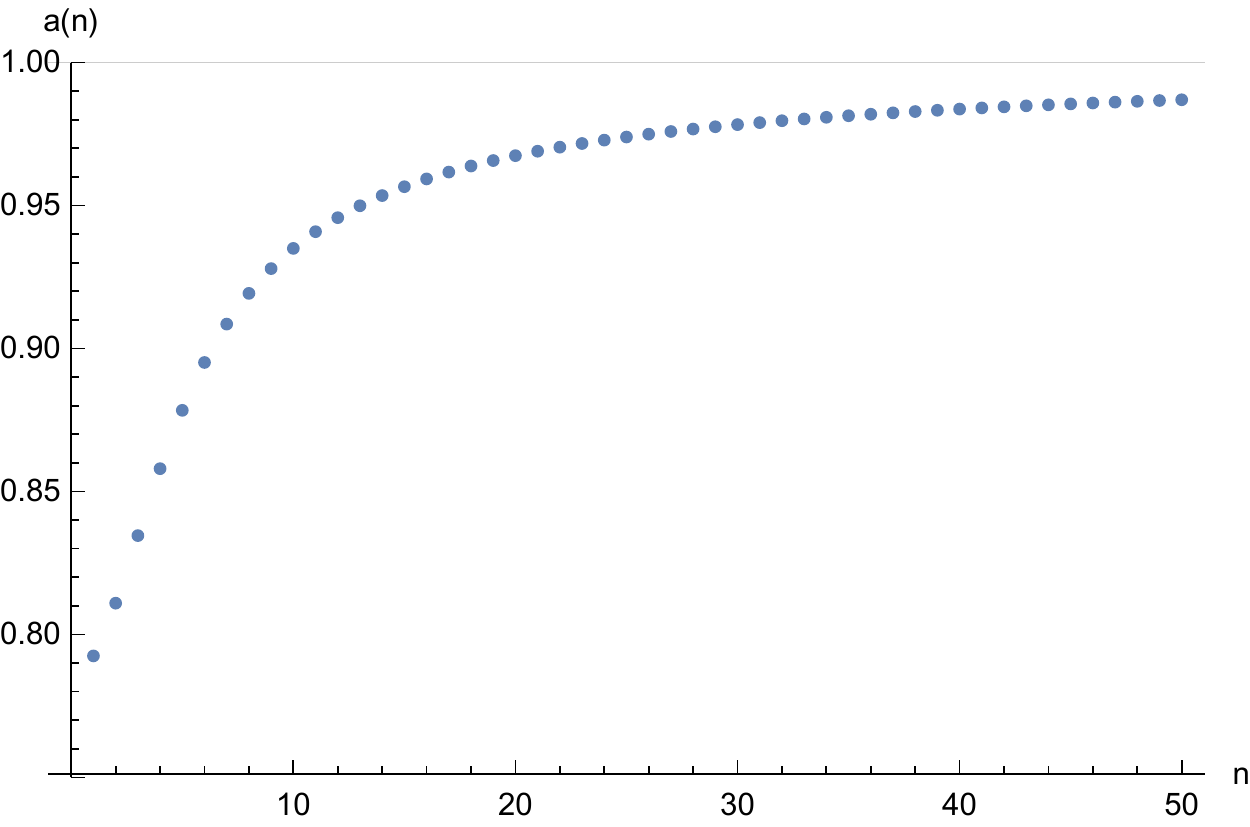}
\caption{Plot of the exponent $a(n)$ for $1\leq n\leq 50$.}
\label{fig:an}
\end{figure}

It is not hard to check that
\begin{equation*}
 \lim_{n\to\infty}a(n)=1.
\end{equation*}
Hence, our discrepancy estimate in Theorem~\ref{thm:genupper} approaches the trivial bound $D_N^*\leq 1$ for huge $n$. However, here we can refer to the result mentioned in Remark~\ref{rem:1} implying that the exponent of $N$ experiences a sudden drop by approximately $1/2$ in the unperturbed case $n=\infty$. 

More generally, for the star discrepancy of two-dimensional sequences it is known that
\begin{equation*}
 \Ds(\cS) \gg N^{-1} (\log N)^{1+\eta},\qquad \eta = 1/(32+4\sqrt{41})-\eps,
\end{equation*}
for all $\eps>0$ and all sequences $\cS$. The  existence of $\eta>0$ is due to a  break-through by Bilyk and Lacey in 2008 (see  \cite{BilOnt083}) and was recently quantified by the second author in \cite{PucOna16}. Furthermore, examples of sequences are known which satisfy the essentially same upper bound, but with $(\log N)^2$.

 Individually, the perturbed Halton sequence as well as the Kronecker sequence are subject to the optimal bound in dimension one, i.e. $ \Ds\ll N^{-1}\log N$, if $c_0=1$ in the perturbing sequence $(c_0,c_1,c_2,\ldots)$ and if, e.g., $\alpha$ has bounded continued fraction coefficients, respectively. Apparently, their interplay reveals a more ambivalent behavior. For more details on the individual sequences and further well established examples  and their discrepancy the reader is referred to the monographs \cites{DicDig10,DrmSeq97,KuiUni74}.

From a metric point of view the situation concerning the distribution of perturbed Halton--Kronecker sequences seems to change completely.
\begin{theorem}
 \label{thm:metric}Let $n\in\bN$. There exist real numbers $\lambda_1(n)$ and $\lambda_2(n)$ with
 \begin{equation}
 \label{eqn:lim}
  0\leq1+\log_{2^n}\lambda_1(n)\leq 1+\log_{2^n}\lambda_2(n)\qquad\text{and}\qquad \lim_{n\to\infty}\left(1+\log_{2^n}\lambda_2(n)\right)=0,
 \end{equation}
 such that for almost all $\alpha\in(0,1)$ (in the sense of the Lebesgue measure) and all $\eps>0$ we have
 \begin{equation*}
N\Ds(z_k(n))\ll_{n,\alpha,\eps} N^{1+\log_{2^n}\lambda_2(n)+\eps},
 \end{equation*}
 and 
 \begin{equation*}
N\Ds(z_k(n))\gg N^{1+\log_{2^n}\lambda_1(n)-\eps},
 \end{equation*}
 Furthermore, upper and lower bounds of the exponents in the estimates from above and below, respectively, for small values of $n$ are given in Figure~\ref{fig:exponents}.
\end{theorem}
\begin{figure}[!htbp]
 \begin{equation*}
 \begin{array}{c|ccccc}
n & 1 & 2 & 3 & 4 & 5 \\
\hline
1+\log_{2^n}\lambda_1(n)	&0.40337	&0.37489	&0.34961	&0.32651	&0.30450\\
1+\log_{2^n}\lambda_2(n) &0.40348	&0.37516	& 0.34962	&0.32672	&0.30599\\
 \end{array}
\end{equation*}
\caption{Approximations of the exponents from Theorem~\ref{thm:metric}.}
\label{fig:exponents}
\end{figure}

\begin{remark}\label{rem:2}
Numerical experiments lead us to the conjecture that the exponents are decreasing in $n$. Moreover, in the limit case $n=\infty$  we know from \cite{LarPro13} that for almost all $\alpha$, every $\eps>0$ and $\delta >0$  
$$1\ll N\Ds(z_k(\infty))\ll_{\alpha,\delta} \left(\log N \right)^{2 +\delta}\ll_\eps N^\eps,$$
in accordance to the behavior of $\lambda_2(n)$. I.e., in the case where the density of $1$'s is extremely sparse, (\ref{eqn:lim}) implicitly shows the optimality of the exponents.
\end{remark}

The above theorems strongly rely on estimates of \emph{lacunary trigonometric products} of the form
\begin{equation}
 \label{eqn:Pi}
  \Pi_{r,\boldsymbol{\gamma}}(\alpha)=\prod_{j=0}^{r-1}\left|\cos\left(2^j\alpha\pi+\frac{\gamma_j\pi}{2}  \right) \right|,
 \end{equation}
 where $\boldsymbol{\gamma}=(\gamma_0,\gamma_1,\gamma_2,\ldots)\in\{0,1\}^{\bN_0}$, $\alpha\in(0,1)$ and $r\in\bN$. Here, the term \emph{lacunary} refers  to the exponential growth of the argument of the cosine function.  Since these are interesting subjects in their own right, we present them in the separate Section~\ref{sec:trigprod}. As a matter of fact,  the quantities $\lambda_1(n)$ and $\lambda_2(n)$ occurring in Theorem~\ref{thm:metric}  stem from  the following metric result.
 \begin{proposition}
\label{prop:trigprodmetric}
Let $n\in\bN$. We have 
\begin{equation}
  \label{eqn:trigmetric1}
\int_{0}^{1}\Pi_{nL,\bfc}(\alpha)\rmd\alpha \leq \big(\mu(n)\big)^{L}
\end{equation}
for every $L\in\bN$ with 
$$\mu(n)=\frac{1}{4^n}\sum_{k=0}^{2^n-1}\left|\cos\left(\frac{(1+2k)\pi}{2^{n+1}}\right)\right|^{-1}.$$ 
Furthermore, there are positive real numbers $\lambda_1(n)$ and $\lambda_2(n)$ such that for every $\eps>0$
\begin{equation}
 \label{eqn:trigmetric2}
 \left(2^{L}\right)^{\log_2\lambda_{1}(n)-\eps}\leq  \int_{0}^{1}\Pi_{nL,\bfc}(\alpha)\rmd\alpha \leq \left(2^{L}\right)^{\log_2\lambda_{2}(n)+\eps}
\end{equation}
for $L>L_0(n,\eps)$.
\end{proposition}

The structure of the remaining paper is as follows: In Section~\ref{sec:genupper} we derive a more generic version of the upper bound for the star discrepancy of the sequence $(z_k(n))_{k\geq 0}$, as the one stated in Theorem~\ref{thm:genupper}, which draws the aforementioned connection to the diophantine approximation type of $\alpha$ (i.e., the number $\sigma$) and to the product \eqref{eqn:Pi}, respectively. Moreover, we provide some lower discrepancy bounds for the sequence  $(\{m_k \alpha\})_{k\geq 0}$ and include further auxiliary results which are relevant for the final proofs of our theorems. Section~\ref{sec:trigprod} provides general  bounds for the lacunary product (\ref{eqn:Pi}) with $\boldsymbol{\gamma}=\bfc$ as well as a proof of Proposition~\ref{prop:trigprodmetric}. In a similar fashion, these already appeared in \cites{AisOnp15,FouMet96,FouSom96}. Finally, we give the proofs of our main theorems in Section~\ref{sec:proofs}.

\begin{remark}
 \label{rem:hybrid}
In principle, hybrid sequences are built by juxtaposing pure sequences to higher dimensional sequences and are the subject of various recent papers \cites{HeKr12,HofJNT,HK11,hklp,HofOne10,Kri12,KrPi13}. Prior to these, hybrid sequences that are built by combining low-discrepancy sequences and (pseudo)-random sequences were suggested by Spanier \cite{spanier} to overcome the curse of dimensionality in quasi-Monte Carlo methods. For results on such hybrid sequences see for example \cites{Niederreiter10c,Niederreiter11b,NiederreiterWinterhof11}. 
\end{remark}
\begin{remark}
A famous and well studied combination of two types of pure low-discrepancy sequences are the Halton--Kronecker sequences (see, for instance, \cites{drmlarHK,HofMet12,LarPro13,NieOnt09,NieImp12}). 
 Combinations of different low-discrepancy sequences are interesting objects as they are candidates for new classes of low-discrepancy sequences and since they often raise intriguing number theoretical problems. The study of Halton--Kronecker sequences, for example, gives rise to the question for a $p$-adic analog of the Thue--Siegel--Roth theorem which was established by Ridout \cite{ridout} and, for instance, to the need of real numbers $\alpha$ that have bounded continued fraction coefficients on the one hand, and an explicitly specifiable binary representation on the other (examples of such numbers were discovered by Shallit \cites{ShaSim79}). Contrary to the Halton--Kronecker sequences, Niederreiter--Kronecker sequences appear to be objects which are particularly hard to study. Qualitative results on their discrepancy can be found in \cite{HK11}. The results obtained in this paper reveal quantitative information of such sequences. 
\end{remark}

\section{General upper and lower discrepancy bounds for perturbed Halton--Kronecker sequences and further auxiliary results}
\label{sec:genupper}

Let us denote by $\|t\|$, $t\in\bR$, the distance of $t$ to the nearest integer, i.e. $\|t\|:=\min\{\{t\},1-\{t\}\}$. Furthermore, we abbreviate $\e(t):=e^{2\pi\rmi t}$.

We begin this section with one of the core estimates for the star discrepancy of $(z_k(n))_{k\geq 0}$ which essentially separates the influence of the sequence $\bfc$ from diophantine properties of $\alpha$ via the product \eqref{eqn:Pi} and a term containing expressions of the form $\|2^\ell h\alpha\|$. Higher dimensional analogues  over $\bZ/p\bZ$ with $p$ prime of the proposition below are known to the authors and are only more technical to derive. But as we do not want to divert the reader's attention from the core issues, we do not state this result in its full generality.

\begin{proposition}
 \label{prop:genupper}Let $n\in\bN$. 
For every irrational $\alpha\in(0,1)$ and for $\bfc$ as given in (\ref{eqn:c}) the star discrepancy of the first $N\geq2$ elements of $(z_k(n))_{k\geq 0}$  satisfies

 \begin{multline}
 \label{eqn:genupper}
  N\Ds(z_k(n)) \ll \frac{N}{K}+\frac{N}{H}\log N + \log^2 N+ 
  \\
  +\sum_{\ell=1}^{\lfloor \log_2 K\rfloor} \sum_{h=1}^{\lfloor H/2^{\ell}\rfloor}\frac{1}{h} 
  \Bigg[\frac{1}{\|2^{\ell}h\alpha  \|}
  +
   \sum_{r=0}^{\lfloor \log_2N\rfloor-\ell}2^r\Pi_{r,\bfc^{(\ell)}}(2^{\ell}h\alpha)\Bigg],
 \end{multline}
 for all positive integers $H,K\leq N$,  where $\bfc^{(\ell)}$ denotes the shifted sequence $(c_{\ell},c_{\ell+1},\ldots)$ and where $\Pi_{r,\bfc^{(\ell)}}$ is defined in \eqref{eqn:Pi}.
\end{proposition}

\begin{proof}
This will immediately follow from Lemma~\ref{lem:genupper_unsimplified} and Lemma~\ref{lem:2additive} below.
\end{proof}

In what follows we denote by $s_{\bfc^{(j)}}(k)$ the weighted sum of digits of $k=k_0+2k_1+2k_2+\cdots$ in base $2$ with weight sequence $\bfc$ shifted by $j\geq1$. I.e.,
$$
s_{\bfc^{(j)}}(k)=k_0c_j+k_1c_{j+1}+k_2c_{j+2}+\cdots.
$$
Notice that this is in fact a finite sum as the dyadic expansion of every integer $k$ is finite.
\begin{lemma}
 \label{lem:genupper_unsimplified}
 Under the assumptions of Proposition~\ref{prop:genupper} we have
 \begin{multline}
 \label{eqn:genupper_unsimplified}
  N\Ds(z_k(n)) \ll \frac{N}{K}+\frac{N}{H}\log N + \log^2N \\ 
  +\sum_{\ell=1}^{\lfloor \log_2 K\rfloor} \sum_{h=1}^{\lfloor H/2^{\ell}\rfloor}\frac{1}{h} \Bigg[\frac{1}{\|2^{\ell}h\alpha \|}
  +\Bigg| \sum_{k=0 }^{\lfloor N/2^{\ell}\rfloor-1} \e\left(2^{\ell} h\alpha k +\frac{  s_{\bfc^{(\ell)}}(k)  }{2}\right) \Bigg| \Bigg].
 \end{multline}
 
 \end{lemma}
 \begin{proof}
 Consider an arbitrary but fixed anchored rectangle $J=[0,\beta)\times[0,\gamma)$ with $\beta\neq 1$ in the unit square. Furthermore, we consider the dyadic expansion of $\beta$
 \begin{equation*}
  \beta=2^{-1}\beta_1+2^{-2}\beta_2+\cdots
 \end{equation*}
with $\beta_i\neq1$ infinitely often. Choose $K\leq N$ and abbreviate $\kappa=\lfloor \log_{2}K \rfloor$. On the basis of this we set $\Sigma_k(\beta)=\sum_{j=1}^{k}\beta_j2^{-j}$ and define the intervals $\cB$ and $J_{\beta}(\ell)$, $1\leq\ell\leq \kappa$, for $\beta_{\ell}=1$ by
 \begin{gather*}
  J_{\beta}(\ell):=\left[\Sigma_{\ell-1}(\beta),  \Sigma_\ell(\beta)\right),  \\
  \cB := \left[\Sigma_{\kappa}(\beta),\Sigma_{\kappa}(\beta)+2^{-\kappa}\right).
 \end{gather*}
In this notation we easily obtain
\begin{multline}
\label{eqn:discrwmax}
 |\cA_N(J)-N\lambda(J)|\leq \sum_{\substack{\ell=1\\ \beta_{\ell}=1}}^{\kappa}\left| \cA_{N}(J_{\beta}(\ell)\times[0,\gamma))-N \lambda(J_{\beta}(\ell)\times[0,\gamma)) \right| \\
 +\max\left\{ \cA_{N}(\cB\times[0,1)),N\lambda(\cB\times[0,1))\right\}.
\end{multline}
Note that $\cB$ is a dyadic interval with volume $1/2^\kappa$, hence, since $C$ is non-singular, we have
\begin{equation*}
 \cA_{N}(\cB\times[0,1))\leq \frac{N}{2^{\kappa}}+1=N \lambda(\cB\times[0,1))+1.
\end{equation*}
Consequently,
\begin{equation}
\label{eqn:max}
\max\left\{ \cA_{N}(\cB\times[0,1)),N\lambda(\cB\times[0,1))\right\}
\leq 
\frac{N}{2^{\kappa}}+1\ll \frac{N}{K}.
\end{equation}

To study the first sum on the right-hand side of (\ref{eqn:discrwmax}), consider a fixed $\ell\leq\kappa$ such that $\beta_{\ell=1}$. Let $\sigma_0+2\sigma_1+\cdots$ be the dyadic expansion of a non-negative integer $\sigma$. By the construction of our sequence it is easy to see that $x_{2\sigma+\rho}\in J_{\beta}(\ell)$, $\rho\in\{0,1\}$, iff 
\begin{equation*}
\begin{cases}
 s_{\bfc^{(1)}}(\sigma)=\sigma_0c_1 + \sigma_1 c_2+\cdots \equiv \beta_1-\rho\pmod{2},  \\
 \sigma_i=\beta_{i+2}\quad \text{for all }~0\leq i \leq \ell-3,~\text{ and}\\
 \sigma_{\ell-2}=0,
 \end{cases}
\end{equation*}
while the digits $\sigma_{\ell-1},\sigma_{\ell},\sigma_{\ell+1},\ldots$ remain arbitrary.

 The above set of conditions is equivalent to
\begin{equation*}
 s_{\bfc^{(1)}}(\sigma)\equiv\beta_1-\rho\pmod{2},\qquad \sigma\equiv R_{\beta,\ell}\pmod{2^{\ell-1}},
\end{equation*}
where $0\leq R_{\beta,\ell}<2^{\ell -1}$ denotes a certain integer. This, in turn, holds if and only if
\begin{equation*}
 \sigma\equiv R_{\beta,\ell} \pmod{2^{\ell-1}},\qquad s_{\bfc^{(\ell)}}\left(\lfloor \sigma/2^{\ell-1}\rfloor \right) \equiv \beta_1-\rho-s_{\bfc^{(1)}}(R_{\beta,\ell})\pmod{2}.
\end{equation*}
It is evident that for any integer $v$ we have $s_{\bfc^{(\ell)}}(v)\equiv a\bmod{2}$ iff
\begin{equation}
\label{eqn:sigmaella}
\Sigma_{\ell,a}(v):=\frac{1}{2}\sum_{z\in\{0,1\}}\e\left( \frac{z}{2}\left( s_{\bfc^{(\ell)}}(v)-a\right) \right)=1,
\end{equation}
and  $\Sigma_{\ell,a}(v)=0$ otherwise. Therefore, we may rewrite the above as
\begin{equation}
\label{eqn:conditions}
 x_{2\sigma+\rho}\in J_{\beta}(\ell)\quad\Longleftrightarrow\quad
 \begin{cases}
  \sigma \equiv R_{\beta,\ell} \pmod{2^{\ell-1}},\text{ and}\\
  \Sigma_{\ell,\beta_1-\rho-s_{\bfc^{(1)}}(R_{\beta,\ell})}\left(\lfloor \sigma/2^{\ell-1} \rfloor \right)=1.
 \end{cases}
\end{equation}
For $\ell$ and $\rho$ as above we introduce the increasing sequence $(\sigma_{k}^{(\ell,\rho)})_{k\geq0}$ composed of all the integers solving  (\ref{eqn:conditions}). Since infinitely many elements of the sequence $\bfc$ are different from 0, this is an infinite sequence. Furthermore, we define the numbers  $S^{(\ell,\rho)}(N)=k_0+1$,  where $2\sigma_{k_0}^{(\ell,\rho)}+\rho <N\leq 2\sigma_{k_0+1}^{(\ell,\rho)}+\rho $.  Since $C$ is non-singular we have
\begin{equation}
\label{eqn:Slrho}
 \lfloor N/2^{\ell} \rfloor \leq S^{(\ell,0)}(N)+S^{(\ell,1)}(N)\leq\lfloor N/2^{\ell}\rfloor+1.
\end{equation}

Let us now continue with (\ref{eqn:discrwmax}). Due to the above discussion we obtain

\begin{multline*} \lefteqn{|\cA_N(J_\beta(\ell)\times[0,\gamma))-N\lambda(J_\beta(\ell)\times[0,\gamma))|}\\
 \leq 1+\sum_{\rho\in\{0,1\}}
\left|\#\left\{ 0\leq \nu<N:\nu\in \{\sigma^{\ell,\rho}_k:{k\geq 0}\}, \{(2\nu+\rho)\alpha\}\in[0,\gamma)\right\}-S^{(\ell,\rho)}(N)\lambda((0,\gamma))\right|\\
\leq 1+\sum_{\rho\in\{0,1\}}S^{(\ell,\rho)}(N) \Ds[S^{(\ell,\rho)}(N)](\{(2\sigma_k^{(\ell,\rho)}+\rho)\alpha\}).
\end{multline*}
Together with (\ref{eqn:max}) this yields
\begin{equation}
\label{eqn:discrsigma}
 |\cA_N(J)-N\lambda(J)|\ll\frac{N}{K}+\log K+ \sum_{\substack{\ell=1 \\ \beta_{\ell}=1}}^{\kappa} \sum_{\rho\in\{0,1\}}S^{(\ell,\rho)}(N) \Ds[S^{(\ell,\rho)}(N)](\{(2\sigma_k^{(\ell,\rho)}+\rho)\alpha\})
\end{equation}
For each positive integer $\ell\leq\kappa$ with $\beta_{\ell}=1$, applying the Erd\H{o}s--Tur{\'a}n inequality with $H\leq N$, we obtain for $\rho\in\{0,1\}$
\begin{equation}
 \label{eqn:erdosturan}
 S^{(\ell,\rho)}(N) \Ds[S^{(\ell,\rho)}(N)](\{(2\sigma_k^{(\ell,\rho)}+\rho)\alpha\})
 \ll
 \frac{S^{(\ell,\rho)}(N)}{\lfloor H/2^{\ell}\rfloor}
 +
 \sum_{h=1}^{\lfloor H/2^{\ell} \rfloor}\frac{1}{h}\left| \sum_{k=0}^{S^{(\ell,\rho)}(N)-1} \e\left(2\sigma_{k}^{(\ell,\rho)}h\alpha\right) \right|.
\end{equation}
In view of (\ref{eqn:Slrho}), we clearly have
\begin{equation}
 \label{eqn:NdivHlogK}
 \sum_{\rho\in\{0,1\}} \sum_{\substack{\ell=1 \\ \beta_{\ell}=1}}^{\kappa} \frac{S^{(\ell,\rho)}(N)}{\lfloor H/2^{\ell}\rfloor} \ll \frac{N}{H}\log K.
\end{equation}
On the other hand,
\begin{multline}
 \label{eqn:thetabetaellrho}
 \left| \sum_{k=0}^{S^{(\ell,\rho)}(N)-1} \e\left(2\sigma_{k}^{(\ell,\rho)}h\alpha\right) \right| 
 \ll
 \left| \sum_{k=0}^{\lfloor N/2^{\ell}\rfloor-\theta_{\beta,\ell,\rho}} \Sigma_{\ell,\beta_1-\rho -s_{\bfc^{(1)}}(R_{\beta,\ell})}(k) \e\left(\left(2^{\ell}k+2R_{\beta,\ell}+\rho\right)h\alpha\right)\right|
 \\=
 \left| \sum_{k=0}^{\lfloor N/2^{\ell}\rfloor-\theta_{\beta,\ell,\rho}} \Sigma_{\ell,\beta_1-\rho -s_{\bfc^{(1)}}(R_{\beta,\ell})}(k) \e\left(2^{\ell}kh\alpha\right)\right|,
\end{multline}
where we have used (\ref{eqn:conditions}) with $k$ taking the role of $\lfloor\sigma/2^{\ell-1}\rfloor$ and with $\Sigma_{\ell,\beta_1-\rho -s_{\bfc^{(1)}}(R_{\beta,\ell})}(k)$ eliminating the undesired instances. Next, we dispose of the dependence on $\rho$ by observing that
\begin{multline}
 \label{eqn:eliminaterho}
  \left| \sum_{k=0}^{\lfloor N/2^{\ell}\rfloor-\theta_{\beta,\ell,\rho}} \Sigma_{\ell,\beta_1-\rho -s_{\bfc^{(1)}}(R_{\beta,\ell})}(k) \e\left(2^{\ell}kh\alpha\right)\right|
  \leq
  1+  \left| \sum_{k=0}^{\lfloor N/2^{\ell}\rfloor-1} \Sigma_{\ell,\beta_1-\rho -s_{\bfc^{(1)}}(R_{\beta,\ell})}(k) \e\left(2^{\ell}kh\alpha\right)\right|
  \\ \leq 
   1+\frac{1}{2}\sum_{z\in\{0,1\}}\left|   \sum_{k=0}^{\lfloor N/2^{\ell}\rfloor-1} \e\left(2^{\ell}kh\alpha+s_{\bfc^{(\ell)}}(k)\frac{z}{2}\right)\right| 
\end{multline}
using (\ref{eqn:sigmaella}) and noting that $\beta_1-\rho-s_{\bfc^{(1)}}(R_{\beta,\ell})$ is an integer. For $z=0$ the inner sum is a geometric sum bounded by $\|2^{\ell}h\alpha\|^{-1}$. The inequality (\ref{eqn:genupper_unsimplified}) now follows from combining this last observtion with (\ref{eqn:discrsigma})--(\ref{eqn:eliminaterho}).
\end{proof}

\begin{lemma}
 \label{lem:2additive}
 Under the assumptions of Proposition~\ref{prop:genupper} we have
 \begin{equation}
 \label{eqn:exptotrigprod}
 \left| \sum_{k=0}^{\lfloor N/2^{\ell}\rfloor-1} \e\left(2^{\ell} kh\alpha +s_{\bfc^{(\ell)}}(k)/2\right) \right|
 \leq 
 \sum_{r=0}^{\lfloor \log_{2}N\rfloor-\ell}2^r\Pi_{r,\bfc^{(\ell)}}(2^{\ell}h\alpha).
\end{equation}
\end{lemma}
\begin{proof}
We shall prove that if $f:\bN_0\rightarrow\bR$ is a $2$-additive function, i.e.
 \begin{equation*}
  f(v_0+2v_1+2^2v_2+\cdots)=f(v_0)+f(2v_1)+f(2^2v_2)+\cdots,\qquad v_i\in\{0,1\},
 \end{equation*}
then
\begin{equation}
\label{eqn:2additive}
 \left|\sum_{v=0}^{V-1} \e( f(v))\right|\leq \sum_{r=0}^{\lfloor\log_2V\rfloor}\prod_{j=0}^{r-1}\left|1+ \e(f(2^j))\right|=\sum_{r=0}^{\lfloor\log_2V\rfloor}2^r\prod_{j=0}^{r-1}\left|\cos\left({\pi f(2^j)}\right)\right|
\end{equation}
for all $V\in\bN$. It is then easy to check that the function
$$
f(v)=2^{\ell}vh\alpha+\frac{s_{\bfc^{(\ell)}(v)}}{2}
$$
is 2-additive, and (\ref{eqn:exptotrigprod}) follows immediately from (\ref{eqn:2additive}). 

To prove (\ref{eqn:2additive}) we expand $V=V_0+2V_1+\cdots+2^{\lfloor \log_2 V\rfloor}V_{\lfloor \log_2 V\rfloor}$, $V_r\in\{0,1\}$ for all $0\leq r\leq \lfloor \log_2 V\rfloor$. Since $f$ is 2-additive we can estimate the sum on the left-hand side as follows
 \begin{multline*}
\left|\sum_{v=0}^{V-1} \e( f(v))\right|
\leq 
\sum_{\substack{r=0 \\ V_r=1}}^{ \lfloor \log_2 V\rfloor}\Bigg|\sum_{k=0}^{2^r-1} \e \Bigg(f(k)+\sum_{j=r}^{ \lfloor \log_2 V\rfloor} f(2^{j}V_j) \Bigg)   \Bigg|\\
\leq 
\sum_{r=0}^{ \lfloor \log_2 V\rfloor} \Bigg| \sum_{k=0}^{2^{r}-1}\e(f(k)) \Bigg|
=
\sum_{r=0}^{ \lfloor \log_2 V\rfloor}\Bigg|  \prod_{j=0}^{r-1}\left(1+\e(f(2^j)) \right)\Bigg|.
 \end{multline*}
 \end{proof}
 
For the actual proofs of our theorems we require results on the term involving $h\|2^{\ell}\alpha h\|$ relying on diophantine properties of $\alpha$. 

\begin{lemma}
 \label{lem:alphafinitetype}
Let $K,H,N$ be positive integers satisfying $K,H\leq N$ and let $\alpha\in\bR$ have bounded continued fraction coefficients. Then 
\begin{equation*}
 \sum_{\ell=1}^{\lfloor \log_2 K \rfloor} \sum_{h=1}^{\lfloor H/2^{\ell} \rfloor} \frac{1}{h\|2^{\ell}\alpha h\|}\ll_{\alpha} H \log K. 
\end{equation*}
Moreover, for almost all $\alpha\in(0,1)$ in the sense of the Lebesgue measure we have
\begin{equation*}
 \sum_{\ell=1}^{\lfloor \log_2 K \rfloor} \sum_{h=1}^{\lfloor H/2^{\ell} \rfloor} \frac{1}{h\|2^{\ell}\alpha h\|}\ll_{\alpha,\eps} N^{\eps} 
\end{equation*}for all $\eps>0$.
\end{lemma}
\begin{proof}
The first claim of this lemma can be found in \cite{drmlarHK}*{Proof of Theorem~2}. The second one is a consequence of \cite{LarPro13}*{Lemma~3}.\end{proof}

\begin{proposition} \label{prop:genlower}
Let $n\in\bN$ and let $N=2^{nL}$ with $L\in\bN$. Then
$$ND_N^*(z_k(n))\geq 2^{nL-3}\Pi_{nL,\boldsymbol{c}}(\alpha)-\frac{|\sin(2^{nL}\pi\alpha)|}{8\sin (\pi\alpha)}.$$
\end{proposition}

\begin{proof}
We use the trivial lower bound that is obtained by specifying the interval under consideration for the first coordinate
\begin{align*}ND_N^*(z_k(n))&=\sup_{0\leq \beta,\gamma \leq 1}\Big|\cA_N(z_k(n),[0,\beta)\times[0,\gamma))-N\lambda_{2}([0,\beta)\times[0,\gamma))    \Big|\\
&\geq \sup_{0\leq \gamma \leq 1}\Big|\cA_N(z_k(n),[0,1/2)\times[0,\gamma))-\frac{N}{2}\lambda_{1}([0,\gamma))    \Big|.
\end{align*}
We now define $(m_k)_{k\geq 0}$ as the increasing sequence of non-negative numbers  satisfying $s_{\boldsymbol{c}}(m_k)\equiv 0\pmod{2}$; or, in other words, let $(m_k)_{k\geq 0}$ be the sequence of indices corresponding to those elements of the perturbed Halton component $(x_k(n))_{k\geq0}$ that lie in the interval $[0,1/2)$.
Then  the above inequality together  with putting $M=N/2=2^{nL-1}$ implies
\begin{equation}\label{ineq:LB1}
ND_N^*(z_k(n))\geq M \Ds[M](\{m_k\alpha\})\geq \frac{1}{4}\left| \sum_{k=0}^{M-1} \e( m_k\alpha)\right|,
\end{equation}
where we used the Koksma--Hlawka inequality in the last step. In what follows we focus on the exponential sum.  
We have
\begin{align*}
  \sum_{k=0}^{M-1} \e(m_k\alpha)
  &=
  \sum_{\substack{m=0 \\ m=\mu_0+2\mu_1+\cdots}}^{2^{nL}-1} \e( m\alpha) \cdot\frac{1}{2}\sum_{z\in\{0,1\}} \e\left(\frac{z}{2} \sum_{j=0}^{nL-1}\mu_{j}c_j \right)\\
    &=
  \frac{1}{2}\sum_{\substack{m=0 \\ m=\mu_0+2\mu_1+\cdots}}^{2^{nL}-1} \e(m\alpha)\cdot \e\left(\frac{1}{2} \sum_{j=0}^{nL-1}\mu_{j}c_j \right)+  \frac{1}{2}\sum_{m=0 }^{2^{nL}-1}\e(m\alpha) 
\end{align*}
The absolute value of the second sum 
can  easily be bounded by $\frac{|\sin(2^{nL}\pi\alpha)|}{2\sin (\pi\alpha)}$ 
and the one of the first sum may be rewritten to yield the estimate
\begin{align}\label{ineq:LB2}
  \left|\sum_{k=0}^{M-1} \e( m_k\alpha)\right|
  &\geq  2^{nL-1} \Pi_{nL,\boldsymbol{c}}(\alpha)-\frac{|\sin(2^{nL}\pi\alpha)|}{2\sin (\pi\alpha)}
\end{align}
\end{proof}

\begin{remark}
 Observe that we have directly linked the discrepancy of $(z_k(n))_{k\geq0}$ to the subsequence $(\{m_k \alpha\})_{k\geq 0}$ of the pure Kronecker sequence via \eqref{ineq:LB1}. If $n=1$, $(m_k)_{k\geq 0}$ translates to  the increasing sequence of non-negative integers with an even sum of digits in base $2$ which are better known as \emph{evil numbers}. The star discrepancy of the associated  \emph{evil Kronecker sequence} with $\alpha$ having bounded continued fraction coefficients has been thoroughly studied in \cite{AisOnp15} and yields the exponents  $\log_4 3\pm\eps$, which coincide with our values  $a(1)\pm\eps$.
\end{remark}

In a recent paper Aistleitner and Larcher focused on metric discrepancy bounds for sequences of the form $(\{a_k \alpha\})_{k\geq 1}$ with $a_k$ growing at most polynomially in $k$. Naturally, this perfectly fits into our setting and we will make use of their result below (see \cite{AisMet16}*{Theorem 3}) for establishing the subsequent Lemma~\ref{lem:LB}, which, in turn, is essential for the proof of Theorem~\ref{thm:metric}.

\begin{lemma}\label{lem:AL}
 Let $(a_k)_{k\geq1}$ be a sequence of integers such that for some $t\in\bN$ we have $a_k\leq k^{t}$ for all $k$ large enough. Assume there exists a number $\tau\in(0,1)$ and a strictly increasing sequence $(B_L)_{L\geq1}$ of positive integers with $(B^{\prime})^{L}\leq B_L \leq B^{L}$ for some reals $B^{\prime},B$ with $1<B^{\prime}<B$, such that for all $\eps>0$ and all $L>L_0(\eps)$ we have 
 \begin{equation*}
  \int_{0}^{1}\left|\sum_{k=1}^{B_L}   \e(a_k\alpha)\right|\rmd\alpha>B_L^{\tau-\eps}.
 \end{equation*}
Then for almost all $\alpha\in[0,1)$ for all $\eps>0$ for the star discrepancy $D^*_N$ of the sequence $(\{a_k\alpha\})_{k\geq1}$ we have
\begin{equation*}
 N D^*_N\gg N^{\tau-\eps}.
\end{equation*}

\end{lemma}

\begin{lemma}\label{lem:LB}
Let $n\in\bN$. If there exists a number $\tau=\tau(n)$ such that for every $\eps>0$ the inequality 
$$\int_{[0,1]}\left(2^{nL}\Pi_{nL,\boldsymbol{c}}(\alpha)-\frac{|\sin(2^{nL}\pi \alpha)|}{\sin(\pi\alpha)}\right)d\alpha\geq 2^{nL(\tau-\eps)}$$
holds for $L$ large enough, then 
$$ND_N^*\gg N^{\tau-\eps}.$$
\end{lemma}
\begin{proof}
This immediately follows from Lemma~\ref{lem:AL} together with the inequalities in \eqref{ineq:LB1} and in \eqref{ineq:LB2}.
\end{proof}

\section{Sharp general and metric estimates for certain lacunary trigonometric products}
\label{sec:trigprod}
To prove Theorems~\ref{thm:genupper} and \ref{thm:genlower}, we need to establish a good uper bound for the trigonometric products $\Pi_{r,\bfc}(\alpha)$ for a wide class of numbers $\alpha$ and also exhibit a specific example to underline the sharpness of our estimate. These are given in Theorem~\ref{thm:trigprod} below. We then focus on metric results for these trigonometric products and establish Proposition~\ref{prop:trigprodmetric}, which is essential for our study of metric discrepancy bounds.

\begin{theorem}
\label{thm:trigprod}
For our periodic perturbing sequence $\bfc$ with period length $n$, as given in (\ref{eqn:c}), we have that for all $\alpha\in[0,1]$, all $r\in\bN$, and all $\ell\in\bN_0$
\begin{equation*}
 \Pi_{r,\bfc^{(\ell)}}(\alpha)\ll_n 2^{-r}\left( \cot \frac{\pi}{2(2^n+1)} \right)^{r/n}.
\end{equation*}
Moreover, this bound is asymptotically optimal in $r$, since for $\ell=0$
$$\Pi_{nL,\bfc}\left(\frac{2^{n-1}}{2^n+1} \right)= 2^{-nL}\left( \cot \frac{\pi}{2(2^n+1)} \right)^{L}.$$
\end{theorem}
 The case $n=1$ has already appeared in \cite{FouSom96}. In this case, two viable strategies are known to treat such products: one by Fouvry and Mauduit \cite{FouSom96} and one by Gel{\cprime}fond \cite{GelSur67}. For our purposes, i.e.  $\bfc$ being of the particular form \eqref{eqn:c}, numerical experiences suggested to pursue the latter.

To this end, we require some notation and initial remarks. We define a system of functions $\{f_\nu:~\nu\geq0\}$ with $f_{\nu}:[0,1]\rightarrow[0,1]$, where
\begin{equation*}
 f_0(x)=x,\qquad f_1(x)=2 x\sqrt{1-x^2},\qquad f_{\nu}=f_1\circ f_{\nu-1}(x),~\nu\geq2.
\end{equation*}
Furthermore, we abbreviate $g(x)=\sqrt{1-x^2}$. We are interested in upper bounds of the function
\begin{equation}
\label{eqn:gsimplified}
G_n=f_0\cdot\prod_{\nu=1}^{n-1}g\circ f_\nu
=
f_{0}\prod_{\nu=1}^{n-1}\sqrt{1-f_{\nu}^2}
=
f_0 \prod_{\nu=1}^{n-1}\frac{f_{\nu+1}}{2f_{\nu}}
=
\frac{f_0 f_n}{2^{n-1}f_1}=
\frac{f_n}{2^n\sqrt{1-f_0^2}}.
\end{equation}
The role of the functions $g$ and $f_{\nu}$ is revealed by taking $x=|\sin y|$. Observe that $g$ now corresponds to a transition to $|\cos y|$ and $f_1$ corresponds to doubling the angle $y$, i.e. $f_1(x)=|\sin(2y)|$. It thus immediately follows that
\begin{equation}
\label{eqn:xi}
 \xi_n=\sin\left( \frac{2^n \pi}{2(2^n+1)}\right)
\end{equation}
is a fixed point of $f_{n}$, i.e.  $\xi_n= f_n(\xi_n)$. This together with (\ref{eqn:gsimplified}) implies
\begin{equation*}
 G_n(\xi_n)= \frac{1}{2^n} \tan\left(\frac{2^n \pi}{2(2^n+1)}  \right)  =\frac{1}{2^n}\cot\left(\frac{\pi}{2(2^n+1)} \right).
\end{equation*}
Moreover, it is an evident observation that $G_n$ and $\xi_n$ are closely related to the trigonometric product and the \emph{bad} $\alpha$ from Theorem~\ref{thm:trigprod}, respectively.
The lemma below generalizes Gel{\cprime}fond's approach.
\begin{lemma}
 \label{lemma:gelfond}
 Let $n\in\bN$ and $\xi_n$ be given as in (\ref{eqn:xi}). For all $x\in[0,1]$ either
 \begin{equation*}
  G_n(x)\leq G_n(\xi_n)\quad\text{or} \quad G_n(x)\left(G_n\circ f_n \right)(x)\leq \big(G_n(\xi_n)\big)^2.
 \end{equation*}
\end{lemma}
\begin{proof}Note that for $n=1$ the result was obtained by  Gel{\cprime}fond \cite{GelSur67} already. 
In the following we concentrate on $n>1$. More precisely,  we verify the first inequality whenever $x\leq \xi_n$ and the second in the case where $x>\xi_n$. We set $x(y)=|\sin (y\pi/2)|$, $y\in[0,1]$,  as well as 
$$
\tilde{G}_n(y)=G_n(x(y))=\frac{\left|\sin\left( 2^n y \pi/2 \right)\right|}{2^n \cos\left( y\pi/2 \right)}.
$$ 
We therefore need to show 
\begin{equation}
\label{eqn:gelfondsimple}
 \tilde{G}_n(y)\leq \tilde{G}_n\left( \frac{2^n}{2^n+1} \right)=G_n(\xi_n),\qquad\text{for all } 0\leq y \leq \frac{2^n}{2^n+1}
\end{equation}
and
\begin{equation}
\label{eqn:gelfondouble}
 \tilde{G}_n(y) \tilde{G}_n(2^ny)\leq \left(G_n(\xi_n)\right)^2,\qquad\text{for all }  \frac{2^n}{2^n+1}<y \leq1.
\end{equation}

Let us first of all focus on (\ref{eqn:gelfondsimple}). This inequality is established by distinguishing between two cases w.r.t. $y$.
\begin{itemize}
\item $y\in[0,(2^n-1)/2^n]$. We use the trivial estimate $$\frac{|\sin(2^ny \pi/2)|}{2^n \cos(y \pi/2)}\leq \frac{1}{2^n\cos((2^n-1)\pi/2^{n+1})}$$ and subsequently show 
$${\cos((2^n-1)\pi/2^{n+1})}\geq \cot(2^n\pi/(2(2^n+1)))$$ 
or, equivalently,
$$\sin(\pi/2^{n+1})\geq \tan(\pi/(2(2^n+1))).$$
To this end we define $z:=1/2^n$ and observe that $z\in[0,1/4]$. We may now rewrite the above inequality as
$$h_1(z):=\sin\left(\frac{z\pi}{2}\right)\geq \tan\left(\frac{z\pi}{2(z+1)} \right)=:h_2(z).$$ 
For $z=0$ we  have equality and for $z\in[0,1/4]$ we observe that $h_1(z)\geq 0$ and $h_2(z)\geq 0$. Moreover, $h_1^{\prime}(z)\geq h_2^{\prime}(z)$ or, equivalently,
$$1\leq \cos\left(\frac{z\pi}{2}\right)(z+1)\cdot\cos^2\left(\frac{z\pi}{2(z+1)}\right)(z+1).$$
Indeed, in what follows we show that each of the two factors above (separated by the dot) is greater or equal to 1.
Let us begin with $\cos(z\pi/2)(z+1)\geq 1$. Equality holds for $z=0$ and the derivative of the left-hand side satisfies
$$-\sin(z\pi/2)(z+1)\pi/2+\cos(z\pi/2)\geq \cos(\pi/8)-\sin(\pi/8)5\pi/8>0$$ 
 whenever $z\in[0,1/4]$. 
 
Similarly,  we have $\cos^2\left(\frac{z\pi}{2(z+1)}\right)(z+1)\geq 1$ for the second factor, since equality holds for $z=0$ and the derivative of the left hand side, i.e.
$$-2\cos\left(\frac{z\pi}{2(z+1)}\right)\sin\left(\frac{z\pi}{2(z+1)}\right)\frac{\pi}{2(z+1)}+\cos^2\left(\frac{z\pi}{2(z+1)}\right),$$
is positive for $z\in[0,1/4]$. This can be derived in the same spirit as above after splitting $[0,1/4]$ into $[0,1/5]$ and $[1/5,1/4]$.   

\item $y\in[(2^n-1)/2^n,2^n/(2^{n}+1)]$. In this case
we write $y=\frac{2^n}{2^{n}+1}-\frac{z}{2^n(2^n+1)}$ with $z\in[0,1]$ and observe 
$$\left|\sin\left(\frac{2^ny \pi}{2}\right)\right|=\sin\left(\frac{2^n\pi}{2(2^n+1)}+\frac{z\pi}{2(2^n+1)}\right).$$
In the following we aim for the inequality 
$$\sin\left(\frac{2^n\pi}{2(2^n+1)}+\frac{z\pi}{2(2^n+1)}\right)/\cos\left(\frac{2^n\pi}{2(2^n+1)}-\frac{z\pi}{2^{n+1}(2^n+1)}\right)\leq\tan\left(\frac{2^n\pi}{2(2^n+1)}\right).$$
We immediately notice that equality holds for $z=0$. Furthermore, we can show that the derivative is negative for $z\in[0,1]$. This is an easy consequence once we have established the inequality
\begin{equation}\label{eq:case2}
\cos\left(\frac{(2^n+z)\pi}{2^{n+1}(2^n+1)}\right)\sin\left(\frac{(2^n+z)\pi}{2(2^n+1)}\right)\geq2^n \cos\left(\frac{(2^n+z)\pi}{2(2^n+1)}\right)\sin\left(\frac{(2^n+z)\pi}{2^{n+1}(2^n+1)}\right)  
\end{equation}
for all $z\in[0,1]$ since, trivially, $\cos\left(\frac{2^n\pi}{2(2^n+1)}-\frac{z\pi}{2^{n+1}(2^n+1)}\right)=\sin\left(\frac{(2^n+z)\pi}{2^{n+1}(2^n+1)}\right)$. First of all we show that the above inequality \eqref{eq:case2} is satisfied for $z=0$. Note that  
\begin{equation*}
 \cos^2\left(\frac{\pi}{2(2^n+1)}\right)\geq 2^n\sin^2\left(\frac{\pi}{2(2^n+1)}\right)\quad\Leftrightarrow\quad1\geq (2^n+1)\sin^2\left(\frac{\pi}{2(2^n+1)}\right).
\end{equation*}
This in turn is the case iff 
\begin{equation*}
 \frac{\eta}{\eta+1}\geq \sin^2\left(\frac{\eta\pi}{2(\eta+1)}\right),
\end{equation*}
where $\eta=1/2^n$ and $\eta\in(0,1/4]$. The last inequality holds as  we have equality for $\eta=0$ and the derivative of the left-hand side is greater than the one of the right-hand side, since $2/\pi\geq \sin(\pi/5)\geq \sin(\eta\pi/(\eta+1))$. 

To finally verify \eqref{eq:case2} for all $z\in[0,1]$ we compute the derivatives of both sides  and observe that the one of the left-hand side oughtweighs the other, since obviously 
$$\sin\left(\frac{(2^n+z)\pi}{2^{n+1}(2^n+1)}\right)\sin\left(\frac{(2^n+z)\pi}{2(2^n+1)}\right)(4^n-1)\geq 0.$$
\end{itemize}
This concludes the proof of (\ref{eqn:gelfondsimple}).

To verify (\ref{eqn:gelfondouble}) we consider an arbitrary but fixed $y\in[2^n/(2^{n}+1),1]$. This interval, in turn, can be parametrized by $z\mapsto 2^n/(2^{n}+1)+z/(4^n(2^n+1))$, $z\in[0,4^n]$. We may now rewrite
\begin{gather*}
\left|{\sin\left(\frac{2^ny\pi}{2}\right)}\right|=\sin\left(\frac{(2^n-z/2^n)\pi}{2(2^n+1)}\right),\qquad
 \left|{\cos\left(\frac{2^ny\pi}{2}\right)}\right|=\cos\left(\frac{(2^n-z/2^n)\pi}{2(2^n+1)}\right),\\
\cos\left(\frac{y\pi}{2}\right)=\cos\left(\frac{(2^n+z/4^n)\pi}{2(2^n+1)}\right).
\end{gather*}
In order to be able to handle $|{\sin(4^ny\pi/2)}|$ we require one further case distinction. 
\begin{itemize}
\item $z\in[0,1]$: Here $|{\sin(4^ny\pi/2)}|=\sin((2^n+z)\pi/(2(2^n+1)))$. We need to derive the following inequality 
\begin{equation}
\label{eqn:case3}
 h_3(z)h_4(z)\leq \tan^2(2^n\pi/(2(2^n+1))), 
 \end{equation}
where
 $h_3(z)=\frac{\sin((2^n+z)\pi/(2(2^n+1)))}{\cos((2^n-z/2^n)\pi/(2(2^n+1)))}$ and $h_4(z)=\frac{\sin((2^n-z/2^n)\pi/(2(2^n+1)))}{\cos((2^n+z/4^n)\pi/(2(2^n+1)))}$. 
Obviously, $0\leq h_3(z)$ and $0\leq h_4(z)$ and for $z=0$ we even have equality in (\ref{eqn:case3}). In the following we show that the derivative of the left-hand side is negative for all $z\in[0,1]$. As a matter of fact, this is a consequence of
\begin{equation*}
 \frac{(h_3(z)h_4(z))'}{h_3(z)h_4(z)}=\frac{h_3^{\prime}(z)}{h_3(z)}+\frac{h_4^{\prime}(z)}{h_4(z)}\leq 0,
\end{equation*}
which in turn can be  rewritten  as
\begin{align*}
\frac{2\cdot 4^n(2^n+1)}{\pi}\frac{(h_3(z)h_4(z))'}{h_3(z)h_4(z)}&=4^n\cot\left(\frac{(2^n+z)\pi}{2(2^n+1)}\right)-2^n \tan\left(\frac{(2^n-z/2^n)\pi}{2(2^n+1)}\right)\\
&\quad -2^n \cot\left(\frac{(2^n-z/2^n)\pi}{2(2^n+1)}\right)+ \tan\left(\frac{(2^n+z/4^n)\pi}{2(2^n+1)}\right) \leq 0.
\end{align*}
Here, we used the identities
\begin{align*}
h_3^{\prime}(z)&= \frac{\pi}{2^{n+1}(2^n+1)}\frac{2^n\cos\left(\frac{(2^n-z/2^n)\pi}{2(2^n+1)}\right)\cos\left(\frac{(2^n+z)\pi}{2(2^n+1)}\right)-\sin\left(\frac{(2^n-z/2^n)\pi}{2(2^n+1)}\right)\sin\left(\frac{(2^n+z)\pi}{2(2^n+1)}\right)}{\cos^2\left(\frac{(2^n-z/2^n)\pi}{2(2^n+1)}\right)},\\
h_4^{\prime}(z)&= -h_3^{\prime}(-z/2^n)/2^n.
\end{align*}
For $z=0$ we have $\frac{(h_3(z)h_4(z))'}{h_3(z)h_4(z)}\leq 0$ due to the proof of  \eqref{eq:case2}.
For arbitrary $z\in(0,1)$ we have  
\begin{align*}
2^n\cot\left(\frac{(2^n+z)\pi}{2(2^n+1)}\right)\left(2^n-\frac{\cot\left(\frac{(2^n-z/2^n)\pi}{2(2^n+1)}\right)}{\cot\left(\frac{(2^n+z)\pi}{2(2^n+1)}\right)}\right) \leq \tan\left(\frac{(2^n-z/2^n)\pi}{2(2^n+1)}\right)\left(2^n-\frac{\tan\left(\frac{(2^n+z/4^n)\pi}{2(2^n+1)}\right)}{\tan\left(\frac{(2^n-z/2^n)\pi}{2(2^n+1)}\right)}\right).
\end{align*}
Indeed, as a consequence of \eqref{eq:case2} we obtain 
$$0\leq 2^n\cot\left(\frac{(2^n+z)\pi}{2(2^n+1)}\right)\leq \tan\left(\frac{(2^n-z/2^n)\pi}{2(2^n+1)}\right).$$
Furthermore, we have
$$2^n-\frac{\cot\left(\frac{(2^n-z/2^n)\pi}{2(2^n+1)}\right)}{\cot\left(\frac{(2^n+z)\pi}{2(2^n+1)}\right)}\leq 2^n-\frac{\tan\left(\frac{(2^n+z/4^n)\pi}{2(2^n+1)}\right)}{\tan\left(\frac{(2^n-z/2^n)\pi}{2(2^n+1)}\right)}$$
since its equivalent version
$$\tan\left(\frac{(2^n+z)\pi}{2(2^n+1)}\right)\geq \tan\left(\frac{(2^n+z/4^n)\pi}{2(2^n+1)}\right)$$
 is obviously satisfied.

It remains to show
$$2^n-\frac{\tan\left(\frac{(2^n+z/4^n)\pi}{2(2^n+1)}\right)}{\tan\left(\frac{(2^n-z/2^n)\pi}{2(2^n+1)}\right)}\geq 2^n-\frac{\tan\left(\frac{(2^n+1/4^n)\pi}{2(2^n+1)}\right)}{\tan\left(\frac{(2^n-1/2^n)\pi}{2(2^n+1)}\right)}\geq 0.$$
The first inequality is evident and for the second one we consider the equivalent formulation which is obtained by setting $\eta:=1/2^n$. I.e.,
\begin{align*}
\frac{(1-\eta)}{\eta}\cos\left(\frac{\eta\pi}{2}\right)\sin\left(\frac{\eta(1-\eta)\pi}{2}\right)\geq \sin\left(\frac{\eta^2\pi}{2}\right).
\end{align*}
This inequality is satisfied for $\eta=0$ as well as for $\eta=1/4$. The right-hand side is monotonically increasing on $[0,1/4]$, while the left-hand side is decreasing, as both $(1-\eta)^2\cos(\eta\pi/2)$ as well as $\sin(\eta(1-\eta)\pi/2)/(\eta(1-\eta))$ are decreasing. 
\item $z\in[1,4^n]$: We exploit the trivial fact   $|{\sin(4^ny\pi/2)}|\leq1$ and, hence, it remains to show that
$$\tan\left(\frac{(2^n-z/2^n)\pi}{2(2^n+1)}\right)\frac{1}{\cos\left(\frac{(2^n+z/4^n)\pi}{2(2^n+1)}\right)}\leq \tan^2\left(\frac{2^n \pi}{2(2^n+1)}\right).$$

For $z=1$ the inequality is true  due to the previous case. Moreover, for $z\to 4^n$   the left-hand side tends to $2^n$. Since $2^n \cos^2\left(\frac{2^n \pi}{2(2^n+1)}\right)\leq \sin^2\left(\frac{2^n \pi}{2(2^n+1)}\right)$ (cf. \eqref{eq:case2}) the sought inequality  is satisfied  for $z=4^n$ too. 
Once again, we need to check whether the left-hand side is decreasing or, equivalently,
$$2^{n+1}\cot\left(\frac{(2^n+z/4^n)\pi}{2(2^n+1)}\right)\geq {\sin\left(\frac{(2^n-z/2^n)\pi}{2^n+1}\right)},\qquad z\in(1,4^n).$$
This is true since we have   equality at the right end point $z=4^n$ and since the derivative of the left-hand side is dominated by the one of the right-hand side, as clearly
\begin{align*}
-\frac{1}{\sin^2\left(\frac{(2^n+z/4^n)\pi}{2(2^n+1)}\right)}&\leq -\cos\left(\frac{(2^n-z/2^n)\pi}{2^n+1}\right).\end{align*}
\end{itemize}
\end{proof}

\begin{proof}[Proof of Theorem~\ref{thm:trigprod}]
First of all, we notice that the implied constant in the sought inequality may depend on $n$. Hence, we can confine ourselves to the case $r\geq2n$, as the claim is trivially fulfilled otherwise.
 Let $j_0$, $0\leq j_0<n$, be the smallest non-negative integer such that $j_0+\ell$ is divisible by $n$. Then we have
 \begin{equation*}
  \Pi_{r,\bfc^{(\ell)}}(\alpha)\leq \prod_{j=j_0}^{r-1}\left| \cos\left( 2^{j}\alpha\pi+\frac{c_{j+\ell}\pi}{2} \right) \right|
  =\prod_{j=0}^{r-j_0-1} \left| \cos\left( 2^{j+j_0}\alpha\pi+\frac{c_j\pi}{2} \right)  \right|.
 \end{equation*}
Assuming $r-j_0=dn+\rho$ with $d\in\bN$ and $0\leq \rho<n$ we obtain further
\begin{align*}
 \Pi_{r,\bfc^{(\ell)}}(\alpha)
 &\leq
 \prod _{j=0}^{dn-1} \left| \cos \left(2^{j+j_0}\alpha\pi+\frac{c_j\pi}{2}\right)   \right| \\
 &=  
 \prod_{\delta=0}^{d-1}\left| \sin \left(2^{\delta n+j_0}\alpha\pi\right)   \right|\cdot \left| \cos \left(2^{\delta n+1+j_0}\alpha\pi\right)   \right| \cdots  \left| \cos \left(2^{ n(\delta+1)+j_0-1}\alpha\pi\right)   \right|\\
 &=\prod_{\delta=0}^{d-1} G_n\left( \left|\sin\left( 2^{\delta n+j_0}\alpha\pi \right)  \right| \right)\\
 &\leq \left(G_n(\xi_n) \right)^{d-1},
\end{align*}
where we used the fact that $\bfc$ has period $n$ in the second and Lemma~\ref{lemma:gelfond} in the last step. The claim now follows as $j_0\leq j_0+\rho< 2n$ and $r/n=d+(j_0+\rho)/n$.

\end{proof}

As it was already mentioned in the beginning of this section we verify the metric estimates for our trigonometric product.
\begin{proof}[Proof of Proposition~\ref{prop:trigprodmetric}]
Following the approaches of \cite{FouSom96} and \cite{AisOnp15} the proof is subdivided into four main steps. First of all, we establish the recurrence relation
 \begin{equation}
 \label{eqn:integralrecursion}
  \int_{0}^{1}\Pi_{nL,\bfc}(\alpha)\rmd\alpha=\int_{0}^{1}\Phi_{n,j}(\alpha) \Pi_{n(L-j),\bfc}(\alpha)\rmd\alpha
 \end{equation}
 with some function $\Phi_{n,j}:[0,1]\rightarrow\bR$, $j\leq L$, which admits the recursive representation
\begin{equation}
\label{eqn:recursion}
 \Phi_{n,j+1}(x)=\frac{1}{2^{n}} \sum_{k=0}^{2^{n}-1} \frac{|\sin(\pi x)|}{2^n |\cos((x+k)\pi/2^n)|} \Phi_{n,j}\left(\frac{x+k}{2^n} \right)=:\frac{1}{2^{n}} \sum_{k=0}^{2^{n}-1} g_n(x,j,k),\quad j\geq0
\end{equation}
with initial value $\Phi_{n,0}\equiv 1$. Secondly, we prove that
\begin{equation}
\label{eqn:symmetric}
 \Phi_{n,j}(x)=\Phi_{n,j}(1-x),
\end{equation}
i.e. $\Phi_{n,j}(x)$ is symmetric about $x=1/2$. As a third step we define
\begin{equation*}
 q_{n,j}(x)=\frac{\Phi_{n,j+1}(x)}{\Phi_{n,j}(x)},\quad M_{n,j}=\max_{0\leq x \leq 1}q_{n,j}(x),\quad\text{and}\quad m_{n,j}=\min_{0\leq x \leq 1}q_{n,j}(x).
\end{equation*}
 and deduce in complete analogy to \cite{AisOnp15} that 
 \begin{equation}
 \label{eqn:monotonicity}
  M_{n,j+1}\leq M_{n,j}\qquad \text{as well as} \qquad m_{n,j}\leq m_{n,j+1}.
 \end{equation}
Finally, we make use of the techniques developed by E.~Foury and C.~Mauduit in \cite{FouMet96} to show that the function $\Phi_{n,1}$ is convex. 

Considering  (\ref{eqn:integralrecursion})--(\ref{eqn:monotonicity}) we can define $\lambda_1(n)=\lim_{j\to\infty} m_{n,j}$ and $\lambda_2(n)=\lim_{j\to\infty}M_{n,j}$, and easily establish the inequality
$$
\lambda_1(n)^{L-k}\int_{0}^{1}\prod_{j=0}^{k-1}q_{n,j}(\alpha)\rmd\alpha \leq \int_{0}^{1}\Pi_{n L,\bfc}(\alpha)\rmd\alpha\leq \lambda_2(n)^{L-k}\int_{0}^{1}\prod_{j=0}^{k-1}q_{n,j}(\alpha)\rmd\alpha
$$
for each $k$. This immediately implies (\ref{eqn:trigmetric2}) and (\ref{eqn:trigmetric1}) follows similarly from (\ref{eqn:monotonicity}) together with the convexity of $\Phi_{n,1}$ by putting $\mu(n)=M_{n,0}=\Phi_{n,1}(1/2)$.

Let us now derive the recurrence (\ref{eqn:integralrecursion}). We do so by demonstrating the first step, i.e. for $j=1$, and the general version follows from iteratively applying the arguments below. Similarly as in \cite{FouSom96}*{(4.1)}, we may rewrite the left-hand side as follows
\begin{align*}
 \int_{0}^{1}\Pi_{n L,\bfc}(\alpha)\rmd\alpha
 &= 
 \int_{0}^{1}  \Pi_{n,\bfc}(\alpha)  \Pi_{n(L-1),\bfc}(2^{n}\alpha) \rmd\alpha 
 =
 \sum_{k=0}^{2^{n}-1}\int_{k/2^{n}}^{(k+1)/2^{n}} \Pi_{n,\bfc}(\alpha)  \Pi_{n(L-1),\bfc}(2^{n}\alpha) \rmd\alpha \\
 &= 
  \sum_{k=0}^{2^{n}-1} \frac{1}{2^{n}}\int_{0}^{1} \Pi_{n,\bfc}\left( \frac{\tilde{\alpha}+k}{2^{n}} \right) \Pi_{n(L-1),\bfc}(\tilde{\alpha}+k)\rmd\tilde{\alpha} \\
 &= \int_{0}^{1} \Phi_{n,1}(\tilde{\alpha})\Pi_{n(L-1),\bfc}(\tilde{\alpha}+k)\rmd\tilde{\alpha},
\end{align*}
where we used the transformation $\tilde{\alpha}=2^{n}\alpha-k$ in the third and the periodicity of $\Pi_{n(L-1),\bfc}$ in the last step, and where we abbreviated
$$
\Phi_{n,1}(\alpha)=\frac{1}{2^{n}}\sum_{k=0}^{2^{n}-1}\Pi_{n,\bfc}\left( \frac{\alpha+k}{2^{n}} \right).
$$
This verifies (\ref{eqn:integralrecursion}). Observe that by repeated applications of the identity $\sin(2x)=2\sin(x)\cos(x)$ we obtain further
$$
\frac{|\sin(\alpha\pi)|}{|\cos(\frac{(\alpha+k)\pi}{2^{n}})|}=  \frac{|\sin((\alpha+k)\pi)|}{|\cos(\frac{(\alpha+k)\pi}{2^{n}})|}  =      \frac{2\left|\sin\left(\frac{(\alpha+k)\pi}{2}  \right)\right|}{|\cos(\frac{(\alpha+k)\pi}{2^{n}})|} = \ldots = 2^{n}\Pi_{n,\bfc}\left( \frac{\alpha+k}{2^{n}} \right),
$$
which is (\ref{eqn:recursion}).
 
For (\ref{eqn:symmetric}) we notice that the relation $g_n(x,j,k)=g_n(1-x,j,2^n-1-k)$ can be proven by induction on $j$ and (\ref{eqn:recursion}) without much effort. It is then easy to see that $\Phi_{n,j}(x)$ is symmetric about $x=1/2$.

To approach (\ref{eqn:monotonicity}) we closely follow the corresponding lines of \cite{AisOnp15}*{Proof of Lemma~7} to see that we have for each $\alpha\in[0,1]$
\begin{multline*}
 q_{n,j}(\alpha)
 =
 \frac{\Phi_{n,j+1}(\alpha)}{\Phi_{n,j}(\alpha)}
 =
 \frac{\sum_{k=0}^{2^{n}-1}\frac{|\sin(\alpha\pi)|}{\left|\cos\left((\alpha+k)\pi/2^{n} \right)\right|} \Phi_{n,j}\left( \frac{\alpha+k}{2^{n}} \right)}{\sum_{k=0}^{2^{n}-1}\frac{|\sin(\alpha\pi)|}{\left|\cos\left((\alpha+k)\pi/2^{n} \right)\right|} \Phi_{n,j-1}\left( \frac{\alpha+k}{2^{n}} \right)}
 \\ \leq
 \frac{\sum_{k=0}^{2^{n}-1}\frac{|\sin(\alpha\pi)|}{\left|\cos\left((\alpha+k)\pi/2^{n} \right)\right|} \Phi_{n,j-1}\left( \frac{\alpha+k}{2^{n}} \right)M_{n,j-1}}{\sum_{k=0}^{2^{n}-1}\frac{|\sin(\alpha\pi)|}{\left|\cos\left((\alpha+k)\pi/2^{n} \right)\right|} \Phi_{n,j-1}\left( \frac{\alpha+k}{2^{n}} \right)}
=
 M_{n,j-1},
\end{multline*}
where we used (\ref{eqn:integralrecursion}) in the second step. Hence, $M_{n,j}\leq M_{n,j-1}$. In the same spirit it is possible to derive $m_{n,j}\geq m_{n,j-1}$.

Let us now focus on the concavity of $\Phi_{n,1}$ using techniques from \cite{FouMet96}.
For $n=1$ this was shown in \cite{FouSom96} and hence we assume $n\geq2$. Furthermore, observe  that
 $$
 2^{2n}\Phi_{n,1}(x)=\sum_{k=0}^{2^n-1}\frac{\sin(\pi x)}{\left| \cos\left(\frac{(x+k)}{2^n}\pi\right) \right|} =\sum_{k=0}^{2^{n-1}-1}\sin(\pi x)\left(  \frac{1}{\cos\left(\frac{(x+k)}{2^n}\pi\right)} + \frac{1}{\cos\left(\frac{(k+1-x)}{2^n}\pi\right)}\right).
 $$
 For $0\leq u\leq 2^{-n}$ and $0\leq k<2^{n-1}$ we introduce the functions
 \begin{equation*}
  \Psi_{k}^{(1)}(u)=\frac{\sin(2^n\pi u)}{\cos\pi\left(u+\frac{k}{2^n}\right)}\qquad\text{and}\qquad \Psi_{k}^{(2)}(u)=\frac{\sin(2^{n}\pi u)}{\cos\pi\left(\frac{k+1}{2^n}-u \right)}.
 \end{equation*}
After the change of variable $x=2^{n}u$ it remains to show that $\sum_{k=0}^{2^{n-1}-1}\left(\Psi_{k}^{(1)}(u)+\Psi_{k}^{(2)}(u)\right)$ is concave. It is immediate that
$$
\Psi_{k}^{(1)}(u)=\frac{(-1)^{k}\sin\left(2^{n}\pi\left(u+\frac{k}{2^n}\right)\right)}{\cos\left(\pi \left(u+\frac{k}{2^n}\right)\right)},\qquad\text{and}\qquad \Psi_{k}^{(2)}(u)=\frac{(-1)^{k}\sin\left(2^{n}\pi\left(\frac{k+1}{2^n}-u\right)\right)}{\cos\left(\pi \left(\frac{k+1}{2^n}-u\right)\right)}.
$$
Using the well-known trigonometric identities
$
\sin(2x)=2\sin(x)\cos(x)$ as well as $\sin(x)\cos(y)=\frac12(\sin(x-y)+\sin(x+y))$
we can inductivley prove that
\begin{equation}
\label{eqn:sintosum}
\frac{\sin2^{n}x}{\cos x}=2\sum_{l=1}^{2^{n-1}}(-1)^{l}\sin((2l-1)x).
\end{equation}
Let us focus on $\Psi_{k}^{(1)}$ first. As a consequence of (\ref{eqn:sintosum}) we may rewrite
\begin{align*}
\sum_{k=0}^{2^{n-1}-1}\Psi_{k}^{(1)}(u)&=2\sum_{l=1}^{2^{n-1}}(-1)^l\sum_{k=0}^{2^{n-1}-1}(-1)^k \sin\left( (2l-1)\pi u+k \frac{2l-1}{2^n}\pi\right)\\
&=\sum_{l=1}^{2^{n-1}}(-1)^l\sum_{k=0}^{2^{n-1}-1}(-1)^k \cos\left( (2l-1)\pi u-\pi/2+k \frac{2l-1}{2^n}\pi\right).
\end{align*}
We invoke the following formula from \cite{FouMet96}*{p.~345}, 
$$
\sum_{k=0}^{m-1}(-1)^{k}\cos(a+hk)=\frac{\cos\left( a+\frac{m-1}{2}h+\frac{m-1}{2}\pi \right) \sin\left(  \frac{mh}{2}+\frac{m\pi}{2}\right)}{\cos\frac{h}{2}}
$$
with $m=2^{n-1}, a=(2l-1)\pi u-\pi/2, h=2^{-n}(2l-1)\pi$ to find that
\begin{multline*}
 \sum_{k=0}^{2^{n-1}-1}\Psi_{k}^{(1)}(u) 
 =2\sum_{l=1}^{2^{n-1}}(-1)^{l}\frac{\sin\left( (2l-1)\pi u+\frac{(2^{n-1}-1)(2l-1)}{2^{n+1}}\pi+\frac{2^{n-1}-1}{2}\pi \right)\sin\left( \frac{(2l-1)}{4}\pi+\frac{2^{n-1}}{2}\pi \right)}{\cos\left( \frac{2l-1}{2^{n+1}}\pi\right)}
 \\=
 2\sum_{l=1}^{2^{n-1}}(-1)^{l+1}\frac{\cos\left( (2l-1)\pi\left(u+\frac{2^{n-1}-1}{2^{n+1}} \right) \right)\sin\left( \frac{2l-1}{4}\pi \right)}{\cos\left(\frac{2l-1}{2^{n+1}}\pi\right)}.
\end{multline*}
Observe that  the simplification of the numerator in the last line follows a different line of reasoning for $n=2$ as for $n\geq3$, yet the result remains the same. Using $\Psi_{k}^{(2)}(u)=\Psi_{k}^{(1)}(1/2^n-u)$ we rewrite
\begin{equation*}
 \sum_{k=0}^{2^{n-1}-1}\Psi_{k}^{(2)}(u)
 =
 2\sum_{l=1}^{2^{n-1}}(-1)^{l+1}\frac{\cos\left( (2l-1)\pi \left( \frac{2^{n-1}+1}{2^{n+1}}-u\right) \right) \sin\left( \frac{2l-1}{4}\pi \right)}{\cos\left(\frac{2l-1}{2^{n+1}}\pi\right)}.
\end{equation*}
Considering the identity $2\cos(x)\cos(y)=\cos(x+y)+\cos(x-y)$ with $x=(2l-1)\pi/4$ and $y=(2l-1)\pi(u-2^{-n-1})$ and, subsequently, $\sin((2l-1)\pi/4)\cos((2l-1)\pi/4)=(-1)^{l+1}/2$ we can simplify as follows
$$
\sum_{k=0}^{2^{n-1}-1}\left(\Psi_{k}^{(1)}(u)+\Psi_{k}^{(2)}(u)\right) = 2 \sum_{l=1}^{2^{n-1}} \frac{\cos\left((2l-1)\pi\left( u-\frac{1}{2^{n+1}}\right) \right)}{\cos\left( \frac{2l-1}{2^{n+1}}\pi\right)}.
$$
Note that $(2l-1)\pi\left( u-\frac{1}{2^{n+1}}\right)\in(-\pi/2,\pi/2)$ and $\frac{2l-1}{2^{n+1}}\pi\in(0,\pi/2)$. Therefore, each summand is a concave function and, hence, so is $\Phi_{n,1}$.
\end{proof}

We want to point out that, since $(M_{n,j})_{j\geq 0}$ is a	 decreasing and $(m_{n,j})_{j\geq 0}$ is an increasing sequence, we are in a position to numerically compute lower and upper bounds for both $\lambda_1(n)$ and $\lambda_2(n)$ for small values of $n$ on the basis of the recurrence relation (\ref{eqn:recursion}). Some approximative values of $1+\log_{2^n}\lambda_i(n)$, $i\in\{1,2\}$, are provided in Figure~\ref{fig:exponents}. It needs to be mentioned that Fouvry and Mauduit ensured that $\lambda_1(1)=\lambda_2(1)$ in \cite{FouSom96} . As our main interest lies in the exponent of the star discrepancy we settle for our approximations at the moment  and keep a generalization of the result of Fouvry and Mauduit for larger $n\in\bN$ for future research.

\section{Proof of the main theorems}
\label{sec:proofs}
\begin{proof}[Proof of Theorem \ref{thm:genupper}]
We begin with the inequality (\ref{eqn:genupper}) from Proposition~\ref{prop:genupper}. Considering Lemma~\ref{lem:alphafinitetype} as well as Theorem~\ref{thm:trigprod} we obtain
\begin{align*}
 N\Ds(z_k(n))&\ll_{n,\alpha}  \frac{N}{K}+{\frac{N}{H}}^{1+\eps}+N^\eps +H \log K+\left( \cot\frac{\pi}{2(2^n+1)} \right)^{\frac{\log_2N}{n}} \log K \log H\\
 &\ll   \frac{N}{K}+{\frac{N}{H}}^{1+\eps}+N^\eps +H \log K +N^{a(n)+\eps}.
\end{align*}
Putting $H=\lfloor\sqrt{N}\rfloor$ and $K=N$ and considering $a(n)\geq \log_4 3\geq 1/2$ finalizes the proof.
\end{proof}

\begin{remark}\label{rem:41}
The result of Theorem~\ref{thm:genupper} may be sharpened by replacing $N^{\eps}$ by a proper power of $\log N$. Moreover, we need to add that it is valid for an even wider class of numbers $\alpha$. Indeed, suppose $\alpha$ is of finite type $\sigma$, i.e. $\|q\alpha\|\geq c_{\alpha,\eps}q^{-\sigma+\eps}$ for all $q\in\bZ\setminus\{0\}$ (see, e.g., \cite{NieOnt09}). For such $\alpha$ the following discrepancy bound can be derived (cf. proof of \cite{NieImp12}*{Theorem~1})
$$N\Ds(z_k(n))\ll_{n,\alpha,\eps} 
  N^{1-1/(\sigma+1)+\eps}+N^{a(n)+\eps}.$$ 
Balancing both terms  yields  a bound on $\sigma$ depending on $n$. Note that almost all $\alpha$ are of finite type $1$, hence Theorem~\ref{thm:genupper} holds  for almost all $\alpha\in(0,1)$ in the sense of the Lebesgue measure as well. Nevertheless, this metric bound is far from being optimal, considering  Theorem~\ref{thm:metric}. 
\end{remark}

\begin{proof}[Proof of Theorem~\ref{thm:genlower}]
We choose $N$ to be of the form $N=2^{nL}$, $L\in\bN$. Subsequently, we refer to Proposition~\ref{prop:genlower} to find that $N\Ds (z_k(n))\geq 2^{nL-1}\Pi_{nL,\bfc}(\alpha)-1/(4\|\alpha\|)$.  In what follows we abbreviate $\alpha_0:=\alpha-\beta$ as well as $\delta_{\ell}:=\{2^{\ell}\beta\}$.  Due to several well known trigonometric identities we may rewrite
\begin{align*}
 |\sin(2^{n\lambda}\alpha \pi)|&=|\sin(2^{n\lambda}\alpha_0 \pi) \cos(\delta_{n\lambda}\pi)\pm\cos(2^{n\lambda} \alpha_0 \pi)\sin(\delta_{n\lambda} \pi)|,\\
 |\cos(2^{n\lambda+\nu}\alpha \pi)| &= |\cos(2^{n\lambda+\nu}\alpha_0 \pi) \cos(\delta_{n\lambda+\nu}\pi)\pm\sin(2^{n\lambda+\nu}\alpha_0\pi) \sin(\delta_{n\lambda+\nu}\pi)|. 
\end{align*}
Using these as well as 
\begin{equation*}
|\sin(2^{n\lambda}\alpha_0 \pi)| = |\cos(\pi/(2^{n+1}+2))|, \qquad\text{and} \qquad  |\cos(2^{n\lambda}\alpha_0 \pi)|=|\sin(\pi/(2^{n+1}+2))|
\end{equation*}
we further obtain
\begin{equation}
\label{eqn:sclambdanu}
2^{nL}\Pi_{nL,\bfc}(\alpha)
= 
N^{a(n)} \Pi_{nL,\bfc}(\alpha)\left(\Pi_{nL,\bfc}(\alpha_0)\right)^{-1}=N^{a(n)}\prod_{\lambda=0}^{L-1}\left(S_{\lambda}\prod_{\nu=1}^{n-1} C_{\lambda,\nu}\right),
\end{equation}
where 
\begin{align*}
 S_{\lambda}&=\left| \cos(\delta_{n\lambda}\pi)\pm \sin(\delta_{n\lambda}\pi) \tan\left( \frac{\pi}{2(2^{n}+1)} \right)\right|, \\
 C_{\lambda,\nu} &= \left| \cos(\delta_{n\lambda+\nu}\pi)\pm\sin(\delta_{n\lambda+\nu}\pi)\tan\left(\frac{2^{\nu}\pi}{2(2^{n}+1)} \right)  \right|.
\end{align*}
Since, trivially, $1-\cos(x)\leq \sqrt{6}x$ and $\sin x\leq x$ for all $x\geq0$ we have
\begin{equation*}
 S_{\lambda}\geq 1-\delta_{n\lambda}\left( \sqrt{6}+\pi \tan\left( \frac{\pi}{2(2^{n}+1)} \right )\right)=: 1- \delta_{n\lambda}c_0(n).
\end{equation*}
A similar argument gives
\begin{equation*}
 C_{\lambda,\nu}\geq 1- \delta_{n\lambda+\nu}\left(  \sqrt{6}+\pi \tan\left(\frac{2^{\nu}\pi}{2(2^{n}+1)} \right)   \right) =: 1-\delta_{n\lambda+\nu}c_\nu(n),\quad 1\leq\nu<n.
\end{equation*}
On the other hand, for fixed $n$ we can define the numbers $\Lambda_0,\Lambda_1,\ldots,\Lambda_L$ by the relations
\begin{equation*}
 \Lambda_0=\inf_{\lambda\geq0}|\sin(2^{n\lambda}\alpha\pi)/\sin(2^{n\lambda}\alpha_0\pi)| 
\end{equation*}
and 
\begin{equation*}
 \Lambda_\nu=\inf_{\lambda\geq 0}|\cos(2^{n\lambda+\nu}\alpha\pi)/\cos(2^{n\lambda+\nu}\alpha_0\pi)|,\qquad 1\leq\nu<n.
\end{equation*}
Due to the special structure of $\beta=\sum_{k\geq0}4^{-2^{k}}$ we know that these numbers are bounded by positive constants from below, as $\inf\left\{|\{2^\ell\alpha\}-\kappa|:\kappa\in\{0,1,1/2\},\ell\in\bN_0\right\}>0$. 
We may thus continue with (\ref{eqn:sclambdanu}) and find a constant $\overline{c}(n)>0$  such that 
$\max\{1-c_\nu(n) x,\Lambda_\nu\}\geq e^{-\overline{c}(n) x}$ for all $x\geq 0$ and every $\nu\in\{0,1,\ldots,n-1\}$. Hence, 
\begin{multline*}
 2^{nL}\Pi_{nL,\bfc}(\alpha) \gg N^{a(n)} \prod_{\lambda=0}^{L-1}\prod_{\nu=0}^{n-1} \max\left\{ (1-\delta_{n\lambda+\nu}c_\nu(n)),\Lambda_\nu\right\} \\ \geq N^{a(n)}\prod_{\ell=0}^{nL-1}e^{-\overline{c}(n)\delta_{\ell}} \geq N^{a(n)} e^{-{c^*}(n)\log nL}\gg N^{a(n)-\eps}, \quad\text{ with } c^*(n)>0,
\end{multline*}
where we used  $\sum_{\ell=0}^K\delta_\ell\leq \tilde{c}\log K$ for an absolute constant $\tilde{c}>0$ and $K$ large enough. 
\end{proof}

For the proof of Theorem~\ref{thm:metric}, we heavily depend on the ideas and strategies developed in \cite{AisOnp15} which were refined and extended in \cite{AisMet16}. 

\begin{proof}[Proof of Theorem~\ref{thm:metric}]

The lower bound can easily be derived by setting $N=2^{nL}$, invoking Lemma~\ref{lem:LB} and applying the inequality  \eqref{eqn:trigmetric2} from Proposition~\ref{prop:trigprodmetric} together with the estimate
$$\int_{[0,1]}\frac{|\sin(2^k\pi \alpha)|}{\sin(\pi \alpha)}\rmd\alpha\ll k,\qquad k\geq1.$$ 

 For the upper bound we set $K=H=N$ in Proposition~\ref{prop:genupper}. In view of the second part of Lemma~\ref{lem:alphafinitetype} it remains to show that
$$\sum_{\ell=1}^{\lfloor\log_2N\rfloor}\sum_{h=1}^{\lfloor N/2^\ell\rfloor}\frac{1}{h}\sum_{r=0}^{\lfloor\log_2N\rfloor-\ell}2^r\Pi_{r,\bfc^{(\ell)}}(2^\ell h\alpha)\ll_{\alpha,\eps,n}N^{\log_{2^n}(\lambda_2(n)) +1+\eps}$$
for all $\eps>0$ and almost all $\alpha\in(0,1)$ in the sense of the Lebesgue measure. 

As a first step  we dispose of the superscript ${(\ell)}$ in $\bfc^{(\ell)}$ by setting $\kappa(\ell)=n-\ell\mod{n}$ and splitting the sum over $r$, which gives
\begin{align*}
\sum_{r=0}^{\lfloor\log_2N\rfloor-\ell}2^r\Pi_{r,\bfc^{(\ell)}}(2^\ell h\alpha)&\ll
2^{\kappa(\ell)} +\sum_{r=\kappa(\ell)}^{\lfloor\log_2N\rfloor-\ell}2^{r-\kappa(\ell)}2^{\kappa(\ell)}\Pi_{r-\kappa(\ell),\bfc^{(0)}}(2^{\ell+\kappa(\ell)}h\alpha)\\
\ll_n &1 +\sum_{r=0}^{\lfloor\log_2N\rfloor-\ell-\kappa(\ell)}2^{r}2^{\kappa(\ell)}\Pi_{r,\bfc}(2^{\ell+\kappa(\ell)}h\alpha)\\
\ll_n &1 +\sum_{j=0}^{\lfloor(\log_2N)/n\rfloor}2^{nj}\sum_{k=0}^{n-1}2^{k+\kappa(\ell)}\Pi_{nj,\bfc}(2^{\ell+\kappa(\ell)}h\alpha).
\end{align*}
Hence,
\begin{multline}
\label{eqn:sumj}
\sum_{\ell=1}^{\lfloor\log_2N\rfloor}\sum_{h=1}^{\lfloor N/2^\ell\rfloor}\frac{1}{h}\sum_{r=0}^{\lfloor\log_2N\rfloor-\ell}2^r\Pi_{r,\bfc^{(\ell)}}(2^\ell h\alpha)\\
\ll_n  (\log N)^2  +\sum_{j=0}^{\lfloor(\log_2N)/n\rfloor}2^{nj}\sum_{\ell=1}^{\lfloor\log_2N\rfloor}\sum_{h=1}^{N}\frac{1}{h}\sum_{k=0}^{n-1}2^{k+\kappa(\ell)}\Pi_{nj,\bfc}(2^{\ell+\kappa(\ell)}h\alpha).
\end{multline}
We fix $\eps>0$ and set $\mu_{n}:=\lceil (1+\log_{2^n}(\lambda_2(n)))^{-1}\rceil$. Proposition~\ref{prop:trigprodmetric} implies 
\begin{equation}\label{eqn:int1}\int_{[0,1]}\left(\sum_{\ell=1}^{jn\mu_n}\sum_{h=1}^{2^{jn\mu_n}}\frac{1}{h}\sum_{k=0}^{n-1}2^{k+\kappa(\ell)}\Pi_{nj,\bfc}(2^{\ell+\kappa(\ell)}h\alpha)\right)d\alpha \leq c(n) \left(2^{nj}\right)^{\log_{2^n}(\lambda_2(n))+\eps/2}\end{equation}
 for all $j>j_0(n,\eps)$, where $c(n)>0$ is an absolute constant only depending on $n$. 
For all positive integers $j$ and for $\eps>0$ we define the events 
$$G_j:=\left\{ \alpha\in(0,1):\sum_{\ell=1}^{jn\mu_n}\sum_{h=1}^{2^{jn\mu_n}}\frac{1}{h}\sum_{k=0}^{n-1}2^{k+\kappa(\ell)}\Pi_{nj,\bfc}(2^{\ell+\kappa(\ell)}h\alpha)>c(n) \left(2^{nj}\right)^{\log_{2^n}(\lambda_2(n))+\eps}\right\}.$$ 
In \eqref{eqn:int1} we have already seen that 
$$\bP(G_j)\leq c(n)\left(2^{nj}\right)^{-\eps/2},\qquad j>j_0(n,\eps).$$ 
Thus, the Borel--Cantelli lemma implies that  for almost all $\alpha\in(0,1)$ we have
 $$\sum_{\ell=1}^{jn\mu_n}\sum_{h=1}^{2^{jn\mu_n}}\frac{1}{h}\sum_{k=0}^{n-1}2^{k+\kappa(\ell)}\Pi_{nj,\bfc}(2^{\ell+\kappa(\ell)}h\alpha)\leq c(n) \left(2^{nj}\right)^{\log_{2^n}(\lambda_2(n))+\eps},\qquad j\geq j_1(n,\eps).$$

Now let $\eps>0$, $N>2^{n\mu_nj_1(n,\eps)}$ and $\alpha \in(0,1)$ such that the above inequality holds. We split the entire sum over $j$ in (\ref{eqn:sumj}) at $M=\lceil\log_{2^n}N/\mu_n\rceil\geq j_1(n,\eps)$ and may thus finalize the proof of the metric upper bound by the estimates
\begin{equation*}
\sum_{j=0}^{M-1}2^{nj}\sum_{\ell=1}^{\lfloor\log_2N\rfloor}\sum_{h=1}^{N}\frac{1}{h}\sum_{k=0}^{n-1}2^{k+\kappa(\ell)}\Pi_{nj,\bfc}(2^{\ell+\kappa(\ell)}h\alpha)
\ll_n 2^{nM}N^{\eps}\ll_n N^{1+\log_{2^n}(\lambda_2(n))+\eps}
\end{equation*}
and 
\begin{multline*}
 \sum_{j=M}^{\lfloor(\log_2N)/n\rfloor}2^{nj}\sum_{\ell=1}^{\lfloor\log_2N\rfloor}\sum_{h=1}^{N}\frac{1}{h}\sum_{k=0}^{n-1}2^{k+\kappa(\ell)}\Pi_{nj,\bfc}(2^{\ell+\kappa(\ell)}h\alpha)\\
 \leq \sum_{j=M}^{\lfloor(\log_2N)/n\rfloor}2^{nj}\sum_{\ell=1}^{jn\mu_n}\sum_{h=1}^{2^{jn\mu_n}}\frac{1}{h}\sum_{k=0}^{n-1}2^{k+\kappa(\ell)}\Pi_{nj,\bfc}(2^{\ell+\kappa(\ell)}h\alpha)\\
 \ll_n  \sum_{j=M}^{\lfloor(\log_2N)/n\rfloor}(2^{nj})^{1+\log_{2^n}(\lambda_2(n))+\eps}
\ll_n N^{{1+\log_{2^n}(\lambda_2(n))+\eps}}.
\end{multline*}

We still need to verify the limit statement in (\ref{eqn:lim}). Evidently,  $\lambda_2(n)\leq \max_{x\in[0,1]}\Phi_{n,1}(x)=\mu(n)=\frac{1}{4^n}\sum_{k=0}^{2^n-1}|\cos((1+2k)\pi/2^{n+1})|^{-1}$ (cf. proof of Proposition~\ref{prop:trigprodmetric}). Therefore, it suffices to show that
 \begin{equation*}
  \lim_{n\to\infty} \log_{2^n}\mu(n)=-1.
 \end{equation*}
 To this end we rewrite
\begin{align*}
\log_{2^n}\Phi_{n,1}(1/2)&=-1 +\log_{2^n}\left(\frac{1}{2^{n}}\sum_{k=0}^{2^{n}-1}\frac{1}{|\cos(\pi(1/2+k)/2^n)|}\right)\\
\\&=-1 +\log_{2^n}\left(\frac{1}{2^{n-1}}\sum_{k=0}^{2^{n-1}-1}\frac{1}{\cos(\pi(1/2+k)/2^n)}\right)\\
&=-1 +\frac{1}{\log2}\log\left(\left(\frac{1}{2^{n-1}}\sum_{k=0}^{2^{n-1}-1}\frac{1}{\cos(\pi(1/2+k)/2^n)}\right)^{1/n}\right).
\end{align*}

Now, obviously 
$$\left(\frac{1}{2^{n-1}}\sum_{k=0}^{2^{n-1}-1}\frac{1}{\cos(\pi(1/2+k)/2^n)}\right)^{1/n}\geq 1.$$
On the other hand, we can make use of the trivial estimate $\sin(\pi x/2)\geq x$ for $x\in[0,1]$ to obtain further
\begin{multline*}
\sum_{k=0}^{2^{n-1}-1}\frac{1}{\cos(\pi(1/2+k)/2^n)}= \sum_{k=0}^{2^{n-1}-1}\frac{1}{\sin(\pi(1/2+k)/2^n)}\\
\leq \sum_{k=0}^{2^{n-1}-1}\frac{1}{(1/2+k)/2^{n-1}}
=2^{n-1}\sum_{k=0}^{2^{n-1}-1}\frac{1}{1/2+k}\leq 2^{n-1}(2+n \log 2)
\end{multline*}
Substituting this in the original expression we thus obtain 
$$\left(\frac{1}{2^{n-1}}\sum_{k=0}^{2^{n-1}-1}\frac{1}{\cos(\pi(1/2+k)/2^n)}\right)^{1/n}\leq n^{1/n} 2^{1/n}\log^{1/n} 2\stackrel{n\to \infty}{\longrightarrow} 1.$$
\end{proof} 

\section*{Acknowledgements}
The authors are extremely grateful to E. Fouvry for pointing out the reference \cite{FouMet96} which helped to modify several arguments in the proof of Proposition \ref{prop:trigprodmetric}, thus avoiding many technicalities.

\end{document}